\documentclass{fim-l}

\usepackage{amsmath}
\usepackage{amssymb}
\usepackage{amscd}
\usepackage{latexsym}
\usepackage{comment}
\usepackage{xypic}
\usepackage{booktabs}
\usepackage{tikz}
\usepackage{youngtab}
\usetikzlibrary{arrows}
\usetikzlibrary{patterns}

\theoremstyle{plain}
\newtheorem{theorem}{Theorem}[section]
\newtheorem{lemma}[theorem]{Lemma}
\newtheorem{proposition}[theorem]{Proposition}
\newtheorem{corollary}[theorem]{Corollary}

\numberwithin{equation}{section}

\theoremstyle{definition}
\newtheorem{definition}[theorem]{Definition}
\newtheorem{example}[theorem]{Example}
\newtheorem{exercise}[theorem]{Exercise}

\newtheorem{remarks}[theorem]{Remarks}

\newcommand{\id}{\operatorname{id}}
\newcommand{\Coker}{\operatorname{Coker}}
\newcommand{\Hom}{\operatorname{Hom}}

\newcommand{\tr}{\operatorname{tr}}
\newcommand{\End}{\operatorname{End}}
\newcommand{\im}{\operatorname{im}}
\newcommand{\wt}{\operatorname{wt}}
\renewcommand{\dim}{\operatorname{dim}}

\newcommand{\C}{{\mathbb{C}}}
\newcommand{\Z}{{\mathbb{Z}}}
\newcommand{\N}{{\mathbb{N}}}

\newcommand\g{\mathfrak{g}}

\newcommand\bv{\mathbf{v}}
\newcommand\bw{\mathbf{w}}
\newcommand\be{\mathbf{e}}

\newcommand\st{\mathrm{st}}
\newcommand\Gr{\mathrm{Gr}}
\newcommand\Fl{\mathrm{Fl}}

\newcommand\None{n_1}
\newcommand\Ntwo{n_2}
\newcommand\nk{n_k}

\newcommand\lecturecite[2]{#2}

\newcounter{x}
\newcounter{y}
\newcounter{z}

\newcommand\xaxis{210}
\newcommand\yaxis{-30}
\newcommand\zaxis{90}

\newcommand\topside[3]{
  \draw[fill=gray, draw=black,shift={(\xaxis:#1)},shift={(\yaxis:#2)},
  shift={(\zaxis:#3)}] (0,0) -- (30:1) -- (0,1) --(150:1)--(0,0);
}

\newcommand\leftside[3]{
  \fill[fill=gray!20!white, draw=black,shift={(\xaxis:#1)},shift={(\yaxis:#2)},
  shift={(\zaxis:#3)}] (0,0) -- (0,-1) -- (210:1) --(150:1)--(0,0);
}

\newcommand\rightside[3]{
  \fill[fill=white, draw=black,shift={(\xaxis:#1)},shift={(\yaxis:#2)},
  shift={(\zaxis:#3)}] (0,0) -- (30:1) -- (-30:1) --(0,-1)--(0,0);
}

\newcommand\cube[3]{
  \topside{#1}{#2}{#3} \leftside{#1}{#2}{#3} \rightside{#1}{#2}{#3}
}

\newcommand\planepartition[1]{
 \setcounter{x}{-1}
  \foreach \a in {#1} {
    \addtocounter{x}{1}
    \setcounter{y}{-1}
    \foreach \b in \a {
      \addtocounter{y}{1}
      \setcounter{z}{-1}
      \foreach \c in {1,...,\b} {
        \addtocounter{z}{1}
        \cube{\value{x}}{\value{y}}{\value{z}}
      }
    }
  }
}

\DeclareMathOperator\bdim{\mathbf{dim}}

\newcommand{\comments}[1]{}

\begin{document}

\author{Alistair Savage}
\thanks{This research is supported by the Natural
Sciences and Engineering Research Council (NSERC) of Canada}

%\subjclass[2000]{Primary: 14C05,17B69}
\date{\today}

\lecturecite{\title{Geometric realizations of
crystals}}{\title{Lectures on geometric realizations of crystals}}

\begin{abstract}
These are notes for a lecture series given at the Fields Institute
Summer School in Geometric Representation Theory and Extended Affine
Lie Algebras, held at the University of Ottawa in June 2009.  We
give an introduction to the geometric realization of crystal graphs
via the quiver varieties of Lusztig and Nakajima.  The emphasis is
on motivating the constructions through concrete examples.  The
relation between the geometric construction of crystals and
combinatorial realizations using Young tableaux is also discussed.
\end{abstract}

\maketitle \thispagestyle{empty}

\tableofcontents

%%%%%%%%%%%%%%%%%%%%%%%%%%%%%%%%%%%%%%%%%%%%%%%%%%%%%%%%%%%%%%%%%%%%
%
\section*{Introduction}
%
%%%%%%%%%%%%%%%%%%%%%%%%%%%%%%%%%%%%%%%%%%%%%%%%%%%%%%%%%%%%%%%%%%%%

These are notes for a lecture series given by the author at the
Fields Institute Summer School in Geometric Representation Theory
and Extended Affine Lie Algebras, held at the University of Ottawa
in June 2009 and organized by the author and Erhard Neher.  Each
section corresponds to an 80 minute lecture given at the school.

The goal of these lectures is to give an introduction to the
geometric realization of crystal graphs via quiver varieties.  The
aim throughout is to adequately motivate the definitions so that the
reader gains an intuition for the constructions.  For this reason,
concrete examples are discussed in detail and mathematical rigor is
sometimes sacrificed in the name of exposition.  The hope is that
after studying these lectures, the reader will have an intuitive
grasp of the theory and several specific examples at hand that will
equip him or her to explore the literature on this subject.
References are given for proofs or arguments that have been omitted.

Section~\ref{S:sec:motivation} is dedicated to motivating the
definitions to follow in the other four lectures.  Several examples
are discussed in detail, noting connections to ideas that appeared
in the lectures of Kamnitzer \cite{Kamnitzer} and Kang \cite{Kang}
at the same summer school. Equipped with these examples, the
discussion of the general theory begins in
Section~\ref{S:sec:quivers} where the notions of quivers and their
representations are introduced.  In Section~\ref{S:sec:LQV}, we
define the Lusztig quiver varieties and the crystal structure on
their sets of irreducible components.  We see how one obtains the
crystal corresponding to half of the quantized enveloping algebra of
a symmetric Kac-Moody algebra.  The lagrangian Nakajima quiver
varieties are introduced in Section~\ref{S:sec:Nak-QV}.  Here we
define the crystal structure on the sets of irreducible components
and obtain the crystals corresponding to irreducible integrable
highest weight representations.  Finally, in
Section~\ref{S:sec:connections} we describe the relationship between
the geometric realizations of crystals using Nakajima quiver
varieties and the well known combinatorial realizations using
tableaux.  In this final section we restrict our attention to the
Lie algebra $\mathfrak{sl}_n$.

\textbf{Prerequisites.} In these notes, we assume a basic knowledge
of Kac-Moody Lie algebras.  In particular, we assume the reader is
familiar with their definitions and the basics of the theory of
irreducible integrable highest weight representations.   We also
presuppose a knowledge of the basic definitions of crystals.
Students in the summer school benefited from a course on this
subject \cite{Kang}.  A more comprehensive treatment can be found in
the book \cite{HK}.  Some background in algebraic geometry would be
helpful in following these notes, but the reader willing to take
some results on faith should be able to follow the presentation.

\textbf{Acknowledgements.} The author would like to thank all of the
participants of the summer school for their enthusiasm and
insightful questions and the other speakers of the summer school
(Vyjayanthi Chari, Joel Kamnitzer, Seok-Jin Kang, Erhard Neher and
Weiqiang Wang) for their interesting lecture series.

\lecturecite{}{\textbf{Summer school lecture notes.} These notes,
along with lectures notes from the other speakers at the University
of Ottawa Fields Institute Summer School on Geometric Representation
Theory and Extended Affine Lie Algebras will eventually appear in
the Fields Institute Monograph Series.  Until then, notes (as well
as video of the lectures) can be found at
\begin{center}
\texttt{http://av.fields.utoronto.ca/video/08-09/geomrep/}.
\end{center}
}

%%%%%%%%%%%%%%%%%%%%%%%%%%%%%%%%%%%%%%%%%%%%%%%%%%%%%%%%%%%%%%%%%%%%
%
\section{Motivating examples} \label{S:sec:motivation}
%
%%%%%%%%%%%%%%%%%%%%%%%%%%%%%%%%%%%%%%%%%%%%%%%%%%%%%%%%%%%%%%%%%%%%

In this first section, we discuss some motivating examples for the
theory that will be introduced in future sections.  We will
explicitly work out various special cases of the general objects we
will introduce later (such as quiver varieties).  Our goal is to
create a collection of concrete examples that will guide our
intuition and serve as motivation for the general definitions to
follow. We begin by considering the following table.

\bigskip
\begin{center}
\begin{tabular}{ccc}
  \toprule
  Geometry & Algebra (Representation Theory) & Combinatorics (Crystals) \\
  \midrule
  \parbox{2.5cm}{\begin{center} Varieties \\ (components)
  \end{center}}
  & Vector space (basis) & Vertex set \\ \\
  Correspondences & \parbox{5cm}{\begin{center} Lie algebra action
  (Chevalley generators) \end{center}} &
  Crystal operators \\
  \bottomrule
\end{tabular}
\end{center}
\bigskip

In \lecturecite{\textbf{Refer to Joel's chapter}}{\cite{Kamnitzer}},
Kamnitzer explained some relations between the first and second
columns in this table. Namely he described certain varieties whose
homology yielded the underlying vector space $V$ of a representation
of a Lie algebra $\g$ and correspondences which produced operators
realizing the action of Chevalley generators of $\g$ on $V$.  Then,
in \lecturecite{\textbf{Refer to Kang's chapter}}{\cite{Kang}}, Kang
explained the passage from the second column to the third column.
More precisely, he described how certain nice bases in
representations of $U_q(\g)$ yield the vertices of a crystal graph
in the $q \to 0$ limit and how the action of the Chevalley
generators is replaced by crystal operators (colored directed edges
of this graph).

In the current chapter, we will describe a general process in which
one can pass from the first column directly to the third.  In
particular, for any (symmetric) Kac-Moody algebra, one can define
what are called quiver varieties and from these we can obtain the
crystal graph directly. The vertex set of the crystal is the set of
irreducible components of the varieties and the crystal operators
are given by natural geometric operators closely related to
correspondences.

We begin by considering a specific example of the construction.  Let
$\g = \mathfrak{sl}_n$, the Lie algebra of $n \times n$ traceless
matrices.  We let $E_{ij}$ be the matrix with $(i,j)$ entry equal to
one and all other entries equal to zero. Then
\[
  \{e_k := E_{k,k+1},\ f_k:=E_{k+1,k}\}_{1 \le k \le n-1}
\]
are the \emph{Chevalley generators} of $\g$.  We let $\mathfrak{h}$
be the \emph{Cartan subalgebra} consisting of traceless diagonal
matrices.  Then $\mathfrak{h}$ has basis $\{h_k:=
E_{k,k}-E_{k+1,k+1}\}_{1 \le k \le n-1}$.

We define $\varepsilon_k \in \mathfrak{h}^*$, by
$\varepsilon_k(E_{l,l})=\delta_{k,l}$ for $1 \le l,k \le n$.  Note
that $\varepsilon_1 + \dots + \varepsilon_n=0$ since we consider
traceless matrices.  The dual space $\mathfrak{h}^*$ has a basis
consisting of \emph{simple roots} $\{\alpha_k := \varepsilon_k -
\varepsilon_{k+1}\}_{1 \le k \le n-1}$.  We also define the
\emph{fundamental weights} $\omega_k = \varepsilon_1 + \dots +
\varepsilon_k$, $1 \le k \le n-1$.  Then
\[
  P := \bigoplus_{k=1}^{n-1} \Z \omega_i\quad \text{and} \quad
  P^+ := \bigoplus_{k=1}^{n-1} \N \omega_i
\]
are the \emph{weight lattice} and \emph{dominant weight lattice}
respectively.  Elements of $P^+$ are called \emph{dominant integral
weights}. Any $w = w_1 \omega_1 + \dots + w_{n-1} \omega_{n-1} \in
P^+$ can be written in the form
\[
  w = \lambda_1 \varepsilon_1 + \dots + \lambda_{n-1}
  \varepsilon_{n-1},\quad \lambda_k = w_k + \dots + w_{n-1}.
\]
Thus $w$ corresponds to a partition
\[
  \lambda(w) = (\lambda_1, \lambda_2, \dots,
  \lambda_{n-1}),\quad \lambda_i \le \lambda_{i+1},\ i=1,\dots,n-1.
\]
In this way, we will often identify the set of partitions of length
$n-1$ with the dominant weight lattice.

Finite-dimensional representations $V$ of $\g$ have a weight space
decomposition
\[
  V = \bigoplus_{\mu \in P} V_\mu,\quad V_\mu = \{v \in V\ |\ h
  \cdot v = \mu(h) v\ \forall\ h \in \mathfrak{h}\}.
\]
The action of $\g$ on itself yields the \emph{adjoint
representation} and the corresponding weight space decomposition
\[
  \g = \mathfrak{h} \oplus \bigoplus_{\alpha \in \Phi}
  \g_\alpha,\quad \dim \g_\alpha=1 \ \forall\ \alpha \in \Phi,\ \dim
  \mathfrak{h}=n-1,
\]
is called the \emph{root space decomposition} of $\g$.  Here $\Phi
\subseteq P$ is the \emph{set of roots} of $\g$.

Irreducible representations of $\g$ are labeled by their highest
weight (an element of $P^+$).  We denote the irreducible
representation of $\g$ of highest weight $w \in P^+$ by $V(w)$, or
$V(\lambda(w))$ when we wish to label it by the corresponding
partiion.  The adjoint representation of $\g$ is isomorphic to
$V(\omega_1 + \omega_{n-1})$ as a $\g$-module.

In \lecturecite{Section~\ref{K:section_crystal} (see
Example~\ref{K:ex:realization of
B(lambda)})}{\cite[Section~4]{Kang}} Kang described a realization of
the crystals of irreducible representations of $\g=\mathfrak{sl}_n$
(or $\g = \mathfrak{gl}_n$) via tableaux (see \cite{HK} for a more
detailed description). Consider the example of $\g =
\mathfrak{sl}_3$ and $V=V(\omega_1 + \omega_2)$ the adjoint
representation.  Let $w=\omega_1 + \omega_2$ be the highest weight
and let $\lambda = \lambda(w) = (2,1)$ be the corresponding
partition. Then $B(\lambda)$ is the set of semistandard tableaux of
shape $\lambda$ and the crystal is as in
Figure~\ref{S:fig:crystal-graph}\lecturecite{ (see
Example~\ref{K:ex:realization of B(lambda)})}{}.

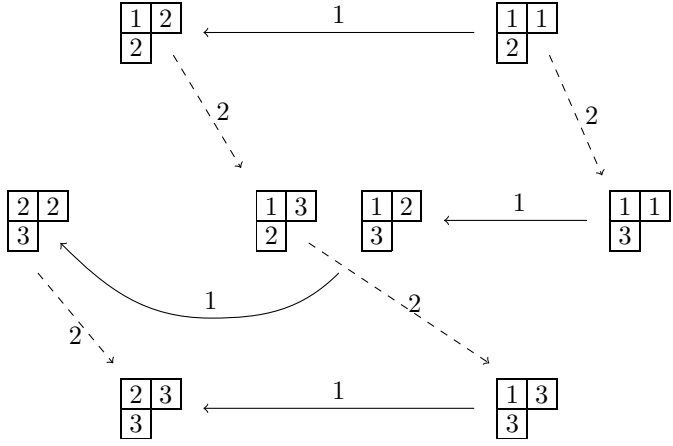
\begin{figure}
\begin{center}
\begin{tikzpicture}
  \draw (-2.5,2.5) node {$\young(12,2)$};
  \draw (2.5,2.5) node {$\young(11,2)$};
  \draw (-0.7,0) node {$\young(13,2)$};
  \draw (0.7,0) node {$\young(12,3)$};
  \draw (4,0) node {$\young(11,3)$};
  \draw (-4,0) node {$\young(22,3)$};
  \draw (-2.5,-2.5) node {$\young(23,3)$};
  \draw (2.5,-2.5) node {$\young(13,3)$};
  \draw[->] (1.8,-2.5) to (0,-2.5) node[above] {1} to (-1.8,-2.5);
  \draw[->] (1.8,2.5) to (0,2.5) node[above] {1} to (-1.8,2.5);
  \draw[->] (3.3,0) to (2.4,0) node[above] {1} to (1.4,0);
  \draw[->] (0,-0.7) to[out=225,in=0] (-1.7,-1.3) node[above] {1}
    to[out=180,in=-45] (-3.7,-0.3);
  \draw[->,dashed] (-2.2,2.2) to (-1.75,1.45) node[right] {2} to (-1.3,0.7);
  \draw[->,dashed] (-0.4,-0.3) to (0.8,-1.1) node[right] {2} to (2,-1.9);
  \draw[->,dashed] (2.8,2.2) to (3.15,1.4) node[right] {2} to (3.5,0.6);
  \draw[->,dashed] (-4,-0.7) to (-3.5,-1.3) node[below] {2} to (-3,-1.9);
\end{tikzpicture}
\end{center}
\caption{Crystal graph of $B((2,1))$, for $\g=\mathfrak{sl}_3$.}
\label{S:fig:crystal-graph}
\end{figure}

Recall that the weight of a tableaux $T$ is equal to $\sum_{i=1}^n
(\text{\# $i$'s in $T$}) \varepsilon_i$.  In the above figure, the
tableaux are grouped according to weight.  The highest weight is $2
\varepsilon_1 + \varepsilon_2 = \omega_1 + \omega_{n-1}$ and the
lowest weight is $\varepsilon_2 + 2 \varepsilon_3 = -\omega_1 -
\omega_{n-1}$.  All weight spaces are one-dimensional except for the
$\varepsilon_1 + \varepsilon_2 + \varepsilon_3=0$ weight space
$\mathfrak{h}$, which is two dimensional.  This corresponds to the
fact that there are two tableaux (in the center of the figure) of
this weight.

We now recall from
\lecturecite{Section~\ref{J:se14}}{\cite[Section~4]{Kamnitzer}}
Ginzburg's construction of irreducible representations of
$\mathfrak{sl}_n$ (or $\mathfrak{gl}_n$) via Springer fibers. For
comparison purposes, let us consider the same example.  That is, let
$\g = \mathfrak{sl}_3$ and let $V = V(\omega_1 + \omega_2)$ be the
adjoint representation as above. We set $w=\omega_1 + \omega_2$ and
$\lambda=\lambda(w)=(2,1)$.  Recall that to construct the
representation $V(w_1 \omega_1 + \dots + w_{n-1} \omega_{n-1})$ of
$\mathfrak{sl}_n$, we fix a nilpotent endomorphism with $w_i$ Jordan
blocks of size $i$.  So we fix a nilpotent $X \in \End(\C^3)$ with
Jordan blocks of size 2 and 1.

Choose the standard basis $\{e_1,e_2,e_3\}$ of $\C^3$ and take
\[
  X = \begin{pmatrix} 0&1&0 \\ 0&0&0 \\ 0&0&0 \end{pmatrix}
\]
in this basis.  Then the variety we are interested in is
\[
  \Fl_3(\C^3)^X = \{0 = V_0 \subseteq V_1 \subseteq V_2 \subseteq
  V_3 = \C^3\ |\ X(V_i) \subseteq V_{i-1},\ 1 \le i \le 3\},
\]
and we have the decomposition
\[
  \Fl_3(\C^3)^X = \bigsqcup_{\mu \in \N^3,\, \mu_1+\mu_2+\mu_3=3}
  \Fl_\mu (\C^3)^X,
\]
where $\Fl_\mu (\C^3)^X$ consists of the flags with $\dim
V_i/V_{i-1} = \mu_i$.

Let $0 = V_0 \subseteq V_1 \subseteq V_2 \subseteq V_3 = \C^3$ be a
flag in $\Fl_\mu(\C^3)^X$.  Note that $X(V_3) = X(\C^3) = \langle
e_1 \rangle$ and thus we must have $\langle e_1 \rangle \subseteq
V_2$. Furthermore $\ker X = X^{-1}(0) = \langle e_1, e_3 \rangle$
and so $V_1 \subseteq \langle e_1, e_3 \rangle$.  Now suppose that
$\mu=(1,1,1)$.  Thus $\dim V_i =i$ for $i=0,1,2,3$.  We see that
$V_1$ can be any one-dimensional subspace of $\langle e_1, e_3
\rangle$ and thus the choice of $V_1$ yields the projective line
$\mathbb{P}^1$.  Now, for a fixed $V_1$, consider the possibilities
for $V_2$.  We must have
\[
  V_1 + \langle e_1 \rangle = V_1 + X(V_3) \subseteq V_2 \subseteq
  X^{-1}(V_1).
\]
Thus, if $V_1=\langle e_1 \rangle$, the only conditions on $V_2$ are
\[
  \langle e_1 \rangle \subseteq V_2 \subseteq \C^3,\quad \dim V_2 =
  2,
\]
and so the choice of $V_2$ yields a projective line $\mathbb{P}^1$.
On the other hand, if $V_1 \ne \langle e_1 \rangle$, then we must
have $V_2 = \langle e_1 \rangle + V_1 = \langle e_1, e_3 \rangle$.
Therefore, $\Fl_{(1,1,1)}(\C^3)^X$ consists of two $\mathbb{P}^1$'s
meeting at a point.

\begin{exercise} \label{S:ex:QV-points}
For all $\mu \in \N^3$ with $\mu_1 + \mu_2 + \mu_3=3$ and $\mu \ne
(1,1,1)$, show that the variety $\Fl_\mu(\C^3)$ is either empty or
is a point.
\end{exercise}

Recall that $H_{\mathrm{top}}(\Fl_3(\C^3)^X) \cong V(\lambda)$ and
that the decomposition $H_{\mathrm{top}}(\Fl_3(\C^3)^X) =
\bigoplus_\mu H_{\mathrm{top}}(\Fl_\mu(\C^3)^X)$ corresponds to the
decomposition $V(\lambda) = \bigoplus_\mu V(\lambda)_\mu$.
Furthermore, a basis of $H_{\mathrm{top}}(\Fl_\mu(\C^3)^X)$ is given
by the fundamental classes of the irreducible components of
$\Fl_\mu(\C^3)^X$.  Using the above and
Exercise~\ref{S:ex:QV-points}, if we draw $\Fl_3(\C^3)^X$ we obtain
Figure~\ref{S:fig:flag-variety}.

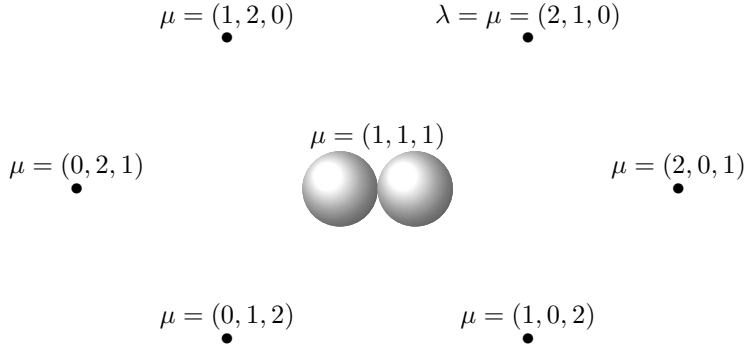
\begin{figure}
\begin{center}
\begin{tikzpicture}
  \draw (-2,2) node[above] {$\mu=(1,2,0)$} node {$\bullet$};
  \draw (2,2) node[above] {$\lambda=\mu=(2,1,0)$} node {$\bullet$};
  \draw (-4,-0) node[above] {$\mu=(0,2,1)$} node {$\bullet$};
  \draw (4,0) node[above] {$\mu=(2,0,1)$} node {$\bullet$};
  \draw (-2,-2) node[above] {$\mu=(0,1,2)$} node {$\bullet$};
  \draw (2,-2) node[above] {$\mu=(1,0,2)$} node {$\bullet$};
  \shade[ball color=white] (-0.5,0) circle (0.5);
  \shade[ball color=white] (0.5,0) circle (0.5);
  \draw (0,0.7) node {$\mu=(1,1,1)$};
\end{tikzpicture}
\bigskip
\caption{$\Fl_3(\C^3)^X$ where $X$ has Jordan type $(2,1)$.}
\label{S:fig:flag-variety}
\end{center}
\end{figure}

Note the similarity between Figures~\ref{S:fig:crystal-graph}
and~\ref{S:fig:flag-variety}.  Our goal is to form a crystal graph
whose vertex set is the set of irreducible components of certain
varieties (such as the Springer fibers discussed above) and whose
crystal operators are defined geometrically.  If we do this, we
should be able to say which of the two $\mathbb{P}^1$'s in
Figure~\ref{S:fig:flag-variety} correspond to which of the weight
zero tableaux in the center of Figure~\ref{S:fig:crystal-graph}.

Let's look at another example.  Consider $\g=\mathfrak{sl}_n$ and
the representation $V=V(N \omega_1)$.  This corresponds to $X=0$.
Then
\[
  \Fl_\mu(\C^N)^X = \Fl_\mu(\C^N) = \{ 0 = V_0
  \subseteq V_1 \subseteq \dots \subseteq V_n = \C^N\ |\ \dim
  V_i/V_{i-1} = \mu_i\}
\]
is irreducible for all $\mu$.  This corresponds to the fact that all
weight spaces of $V$ are one-dimensional.  Therefore the set of
irreducible components of $\Fl_n(\C^N)$ is precisely
\[
  \{ \Fl_\mu(\C^N)\ |\ \mu \in \N^n,\ \mu_1 + \dots + \mu_n = N\}.
\]
We wish to construct crystal operators $\tilde e_i$ and $\tilde f_i$
in a ``natural'' geometric way.  For $i=1,\dots,n-1$, define
\[
  \mu \pm \alpha_i = (\mu_1, \dots, \mu_{i-1}, \mu_i \pm 1,
  \mu_{i+1} \mp 1, \mu_{i+1}, \dots, \mu_n).
\]
Define
\[
  \Fl_{\mu, \mu+\alpha_i} (\C^N) = \{((U_j), (V_j)) \in
  \Fl_\mu(\C^N) \times \Fl_{\mu+\alpha_i}(\C^N)\ |\ U_j \subseteq
  V_j \ \forall\ j\}.
\]
Note that the condition $U_j \subseteq V_j$ implies that $U_j = V_j$
for all $j \ne i$.  Then we have natural projections
\[
  \Fl_\mu(\C^N) \xleftarrow{\pi_1} \Fl_{\mu,\mu+\alpha_i}(\C^N)
  \xrightarrow{\pi_2} \Fl_{\mu + \alpha_i} (\C^N)
\]
given by $\pi((U_j),(V_j)) = (U_j)$ and $\pi_2((U_j),(V_j)) =
(V_j)$.

What are the fibers of the maps $\pi_1$ and $\pi_2$?  Fix $(U_j) \in
\Fl_\mu(\C^N)$.  Then
\begin{align*}
  \pi_1^{-1}((U_j)) &\cong \{V_i\ |\ U_i \subseteq V_i \subseteq
  U_{i+1},\ \dim V_i/U_i = 1\} \\
  &\cong \{\bar V \subseteq U_{i+1}/U_i\ |\ \dim \bar V = 1\} \\
  &\cong \mathbb{P}^{\mu_{i+1}-1}.
\end{align*}
So $\pi_1$ is a fiber bundle with smooth fibers isomorphic to
$\mathbb{P}^{\mu_{i+1}-1}$.

\begin{exercise} \label{S:ex:fiber-bundle}
Show that $\pi_2$ is a fiber bundle with smooth fibers isomorphic to
$\mathbb{P}^{\mu_i-1}$.
\end{exercise}

Since $\pi_1$ and $\pi_2$ are both fiber bundles with smooth fibers,
they induce the following bijections (provided $\mu + \alpha_i \in
\N^n$).

\medskip
\begin{center}
\begin{tabular}{rcccl}
  \parbox{3.3cm}{\begin{center} Irreducible components of $\Fl_\mu
  (\C^N)$ \end{center}} &
  $\stackrel{\text{1-1}}{\longleftrightarrow}$ &
  \parbox{3.3cm}{\begin{center} Irreducible components of
  $\Fl_{\mu,\mu+\alpha_i} (\C^N)$ \end{center}} &
  $\stackrel{\text{1-1}}{\longleftrightarrow}$ &
  \parbox{3.3cm}{\begin{center} Irreducible components of $\Fl_{\mu +
  \alpha_i} (\C^N)$
  \end{center}}
\end{tabular}
\end{center}
\medskip

We can then use these bijections to define the actions of the
crystal operators $\tilde e_i$ and $\tilde f_i$, $i=1, \dots, n-1$.
Namely, $\tilde e_i$ sends an irreducible component of
$\Fl_\mu(\C^N)$ to the corresponding irreducible component of
$\Fl_{\mu+\alpha_i} (\C^N)$ and $\tilde f_i$ does the opposite.

This example will motivate the more general construction in
Section~\ref{S:sec:Nak-QV}.  However, we will need to do some extra
work because the above example was a bit too simple.  In particular,
the weight spaces were all one-dimensional and so the varieties
involved were all irreducible.  We will need to deal with
representations whose weight spaces have higher dimension and thus
whose corresponding varieties are not irreducible.  This will
require us to develop slightly more sophisticated operators.

%%%%%%%%%%%%%%%%%%%%%%%%%%%%%%%%%%%%%%%%%%%%%%%%%%%%%%%%%%%%%%%%%%%%
%
\section{Quivers} \label{S:sec:quivers}
%
%%%%%%%%%%%%%%%%%%%%%%%%%%%%%%%%%%%%%%%%%%%%%%%%%%%%%%%%%%%%%%%%%%%%

The varieties we will use in our geometric construction of crystal
bases are certain varieties attached to quivers.  In this section we
review some of the basic theory of quivers and their
representations.

A \emph{quiver} is simply another name for a directed graph.  Thus,
a quiver is a quadruple $Q = (Q_0,Q_1,s,t)$ where $Q_0$ and $Q_1$
are sets and $s$ and $t$ are maps from $Q_1$ to $Q_0$.  We call
$Q_0$ and $Q_1$ the sets of \emph{vertices} and \emph{directed
edges} (or \emph{arrows}) respectively.  The maps $s$ and $t$ tell
us the endpoints of each arrow: for an arrow $a \in Q_1$, we call
$s(a)$ the \emph{source} of $a$ and $t(a)$ the \emph{target} of $a$.
Usually we will write $Q=(Q_0,Q_1)$, leaving the maps $s$ and $t$
implied. The quiver $Q$ is said to be \emph{finite} if $Q_0$ and
$Q_1$ are both finite.  A \emph{loop} is an arrow $a$ with
$s(a)=t(a)$. In this paper, all quivers will be assumed to be finite
and without loops.  If we forget the orientation of the edges in a
quiver, we obtain a graph, called the \emph{underlying graph} of the
quiver.  A quiver is said to be of \emph{finite type} its underlying
graph is a Dynkin diagram of finite $ADE$ type. Similarly, it is of
\emph{affine} (or \emph{tame}) \emph{type} if the underlying graph
is a Dynkin diagram of affine type and of \emph{indefinite} (or
\emph{wild}) \emph{type} if the underlying graph is a Dynkin diagram
of indefinite type.

A \emph{path} in $Q$ is a sequence of arrows lining up tip-to-tail.
More precisely, a path is a sequence $\beta = a_l a_{l-1} \cdots
a_1$ of arrows such that $t(a_i) = s(a_{i+1})$ for $1 \le i \le
l-1$.  We call $l$ the \emph{length} of the path. We let $s(\beta) =
s(a_1)$ and $t(\beta) = t(a_l)$ denote the initial and final
vertices of the path $\beta$. For each vertex $i \in I$, we have a
trivial path $e_i$ with $s(e_i) = t(e_i)=i$.

The \emph{path algebra} $\C Q$ associated to a quiver $Q$ is the
$\C$-algebra whose underlying vector space has basis the set of
paths in $Q$, and with the product of paths given by concatenation.
More precisely, if $\beta=a_l \cdots a_1$ and $\beta' = b_m \cdots
b_1$ are two paths in $Q$, then $\beta \beta'=a_l \cdots a_1 b_m
\cdots b_1$ if $t(\beta')=s(\beta)$ and $\beta \beta'=0$ otherwise.
This is an associative multiplication. There is a natural grading
$\C Q = \bigoplus_{n \in \N} (\C Q)_n$ where $(\C Q)_n$ is the span
of the paths of length $n$.  This is a grading as an algebra in the
sense that $(\C Q)_n (\C Q)_m \subseteq (\C Q)_{n + m}$ (in fact, we
have equality here).

\begin{example} \label{S:eg:Jordan}
Let $Q$ be the \emph{Jordan quiver}:

\medskip
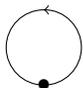
\begin{center}
\begin{tikzpicture}
  \draw[->] (0,0) node {$\bullet$};
  \draw[->] (0,0) arc (-90:90:0.5);
  \draw (0,1) arc (90:270:0.5);
\end{tikzpicture}
\end{center}

Then $\C Q = \C[t]$, the polynomial algebra in one variable.  The
variable $t$ corresponds to the unique path of length one.
\end{example}

\begin{exercise} \label{S:ex:An}
Let $Q$ be the following \emph{quiver of type $A_n$}:

\medskip
\begin{center}
\begin{tikzpicture}
  \draw[-<] (0,0) node {$\bullet$} node[below] {1} -- (0.5,0);
  \draw (0.5,0) -- (1,0);
  \draw[-<] (1,0) node {$\bullet$} node[below] {2} -- (1.5,0);
  \draw (1.5,0) -- (2,0);
  \draw[-<] (2,0) node {$\bullet$} node[below] {3} -- (2.5,0);
  \draw (2.5,0) -- (3,0) node {$\bullet$} node[below] {4};
  \draw (3,0) -- (3.5,0) node[right] {$\cdots$};
  \draw (4.5,0) -- (5,0);
  \draw[-<] (5,0) node {$\bullet$} -- (5.5,0);
  \draw (5.5,0) -- (6,0);
  \draw[-<] (6,0) node {$\bullet$} -- (6.5,0);
  \draw (6.5,0) -- (7,0) node {$\bullet$} node[below] {$n$};
\end{tikzpicture}
\end{center}
\medskip

Then for all $1 \le i \le j \le n$, there exists a unique path
$p_{ij}$ from $j$ to $i$.  Define a map $f : \C Q \to M_{n \times
n}(\C)$, the algebra of $n \times n$ matrices with complex entries,
by $p_{ij} \mapsto E_{ij}$ and extending by linearity. Show that $f$
is an isomorphism (of algebras) onto the algebra of upper triangular
matrices.
\end{exercise}

A \emph{representation} $(V,x)$ of a quiver $Q$ is a collection
\[
  \{V_i\ |\ i \in Q_0\}
\]
of finite-dimensional vector spaces, together with a collection
\[
  \{x_a : V_{s(a)} \to V_{t(a)}\ |\ a \in Q_1\}
\]
of linear maps.  Representations of $Q$ are equivalent to
representations of the path algebra $\C Q$.

We will often view the collection $\{V_i \ |\ i \in Q_0\}$ as an
$Q_0$-graded vector space $V=\bigoplus_{i \in Q_0} V_i$.  Then the
\emph{graded dimension} of $V$ is
\[
  \bdim V = (\dim V_i)_{i \in Q_0}.
\]

A morphism $(V,x) \to (W,y)$ between two representations of a quiver
$Q$ is a collection
\[
  \{\psi_i : V_i \to W_i\ |\ i \in Q_0\}
\]
of linear maps such that the diagram
\[
  \xymatrix{
    V_{s(a)} \ar[r]^{x_a} \ar[d]_{\psi_{s(a)}} & V_{t(a)}
    \ar[d]^{\psi_{t(a)}} \\
    W_{s(a)} \ar[r]^{y_a} & W_{t(a)}
  }
\]
commutes for all $a \in Q_1$.

Quivers and path algebras play an important role in the
representation theory of finite-dimensional associative algebras. In
particular, we have the following result.

\begin{proposition} \label{S:prop:fd-alg}
Let $A$ be a finite-dimensional $\C$-algebra.  Then the category of
representations of $A$ is equivalent to the category of
representations of $\C Q/I$ for some quiver $Q$ and some two-sided
ideal $I$ of $\C Q$.
\end{proposition}

\begin{proof}
This follows from results in \cite{Gab72,Gab71}.  The full proof can
also be found in \cite[Theorem~3.7]{ASS06}.
\end{proof}

Suppose $(V,x)$ and $(W,y)$ are representations of a quiver $Q$.
Then we define the \emph{direct sum} of $(V,x)$ and $(W,y)$ to be
the representation $(V \oplus W, x \oplus y)$ where $(V \oplus W)_i
= V_i \oplus W_i$ for $i \in Q_0$, and $(x \oplus y)_a = x_a \oplus
y_a$ for $a \in Q_1$.

If $(V,x)$ is a representation of a quiver $Q$ and $W$ is an
$Q_0$-graded subspace of $V$ (that is, $W_i \subseteq V_i$ for all
$i \in Q_0$) then we say that $W$ is \emph{$x$-invariant} if
$x_a(W_{s(a)}) \subseteq W_{t(a)}$ for all $a \in Q_1$.  Then $(W,
x|_W)$ is called a \emph{subrepresentation} of $(V,x)$.  The
representation $(0,0)$ is called the \emph{trivial representation}.

A quiver representation $Q$ is said to be \emph{irreducible} (or
\emph{simple}) if it contains no nontrivial proper
subrepresentations. It is called \emph{indecomposable} if cannot be
written as a direct sum of two nontrivial subrepresentations.

\begin{example} \label{S:eg:simples}
Fix a vertex $i \in Q_1$ and let $S^i$ be a $Q_0$-graded vector
space with $S^i_i=\C$ and $S^i_j = 0$ for $i \ne j$.  Then $(S^i,0)$
is a simple representation.
\end{example}

\begin{example} \label{S:eg:An-indecomps}
Let $Q$ be the quiver of Exercise~\ref{S:ex:An}.  For $1 \le i \le j
\le n$, define $V^{i,j}$ to be the $Q_0$-graded vector space defined
by
\[
  V^{i,j}_k = \begin{cases}
    \C & i \le k \le j, \\
    0 & \text{otherwise}.
  \end{cases}
\]
For $a \in Q_1$, define
\[
  x^{i,j}_a = \begin{cases}
    \id & i+1 \le s(a) \le j, \\
    0 & \text{otherwise}.
  \end{cases}
\]
Then $(V^{i,j},x^{i,j})$ is indecomposable.  However, if $i < j$,
then $(V^{i,j},x^{i,j})$ is not irreducible.
\end{example}

Forgetting the orientation of the edges of a quiver $Q$, we obtain
the \emph{underlying graph} $\Gamma(Q)$ of $Q$.  We then let $\g(Q)$
be the Kac-Moody algebra whose Dynkin graph is the graph
$\Gamma(Q)$. Note that such a Kac-Moody algebra has symmetric Cartan
matrix and any Kac-Moody algebra with symmetric Cartan matrix arises
as $\g(Q)$ for some (in general not unique) quiver $Q$.

\begin{theorem}[Gabriel's Theorem \cite{BGP73,Gab72}] \label{S:thm:Gabriel}
A quiver $Q$ has finitely many indecomposable representations (up to
isomorphism) if and only if $Q$ is of finite type.  Furthermore, if
$Q$ is of finite type, there is a one-to-one correspondence between
isomorphism classes of indecomposable representations of $Q$ and
positive roots of $\g(Q)$.  This correspondence is given by
\[
  (V,x) \mapsto \sum_{i \in Q_0} (\dim V_i) \alpha_i,
\]
where the $\alpha_i$, $i \in Q_0$, are the simple roots of $\g(Q)$.
\end{theorem}

\begin{proof}
This result was first proved by Gabriel in \cite{Gab72}.  The proof
is computational and the relation with Dynkin diagrams is a
consequence of this proof and not an important feature of the
calculations.  A second proof was given by Bern{\v{s}}te{\u\i}n,
Gel$'$fand, and Ponomarev in \cite{BGP73} that involves the
machinery of Coxeter functors and Weyl groups.
\end{proof}

\begin{exercise}[Type $A_n$] \label{S:ex:indecomps-An}
Let $Q$ be the quiver of type $A_n$ given in Exercise~\ref{S:ex:An}.
Then $\g(Q)=\mathfrak{sl}_{n+1}$.  Show that the representations
$(V^{i,j},x^{i,j})$ are the only indecomposable representations of
$Q$ up to isomorphism.  The indecomposable representation $(V^{i,j},
x^{i,j})$ corresponds to the positive root
\[
  \alpha_i + \alpha_{i+1} + \dots + \alpha_j
\]
which is the weight of the root space $\C E_{i,j+1}$ of
$\mathfrak{sl}_{n+1}$, where $E_{i,j+1}$ is the elementary matrix
with a one in entry $(i,j+1)$ and a zero in all other entries.
\end{exercise}

\begin{example}[Type $A_n^{(1)}$] \label{S:eg:affine-An-quiver}
Let $Q$ be the quiver
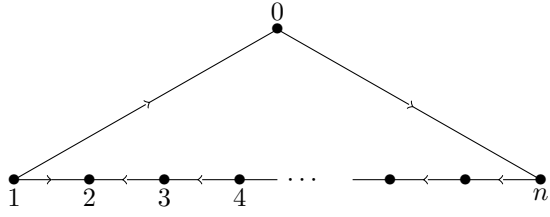
\begin{center}
\begin{tikzpicture}
  \draw[->] (0,0) node {$\bullet$} node[below] {1} -- (0.5,0);
  \draw (0.5,0) -- (1,0);
  \draw[-<] (1,0) node {$\bullet$} node[below] {2} -- (1.5,0);
  \draw (1.5,0) -- (2,0);
  \draw[-<] (2,0) node {$\bullet$} node[below] {3} -- (2.5,0);
  \draw (2.5,0) -- (3,0) node {$\bullet$} node[below] {4};
  \draw (3,0) -- (3.5,0) node[right] {$\cdots$};
  \draw (4.5,0) -- (5,0);
  \draw[-<] (5,0) node {$\bullet$} -- (5.5,0);
  \draw (5.5,0) -- (6,0);
  \draw[-<] (6,0) node {$\bullet$} -- (6.5,0);
  \draw (6.5,0) -- (7,0) node {$\bullet$} node[below] {$n$};
  \draw[-<] (7,0) -- (5.25,1);
  \draw (5.25,1) -- (3.5,2) node {$\bullet$} node[above] {0};
  \draw[-<] (3.5,2) -- (1.75,1);
  \draw (1.75,1) -- (0,0);
\end{tikzpicture}
\end{center}
and consider the positive imaginary root $\delta = \sum_{i=0}^n
\alpha_i$.  Then there exists a one parameter family of
non-isomorphic representations where $\dim V_i =1$ for
$i=0,\dots,n$, and all $x_a$ are non-zero.  The parameter is the
composition around the loop.
\end{example}

\begin{theorem}[Kac's Theorem] \label{S:thm:Kac}
Let $Q$ be an arbitrary quiver.  Then the graded dimensions of
indecomposable representations of $Q$ correspond, via the map
\[
  \bdim V \mapsto \sum_{i \in Q_0} (\dim V_i) \alpha_i,
\]
to the positive roots of the root system of $\g(Q$).
\end{theorem}

\begin{remarks} \label{S:rem:Kac}
\begin{enumerate}
  \item In Kac's Theorem, no assertion is made that positive roots
  are in one-to-one correspondence with isomorphism classes of
  indecomposable representations.

  \item Real roots correspond to graded dimensions with one
  indecomposable representation of that dimension and imaginary
  roots correspond to graded dimensions with families of
  indecomposable representations of that dimension.
\end{enumerate}
\end{remarks}

Assume for simplicity that $Q$ is of finite type.  Then any
representation can be written in a unique way (up to isomorphism) as
a sum of indecomposable representations labeled by positive roots of
$\g=\g(Q)$.  Let $\Phi$ be the set of roots of $\g$ and $\Phi^\pm$
the set of positive/negative roots.  Recall that $\g$ has a
\emph{triangular decomposition}
\[ \textstyle
  \g = \mathfrak{n}^- \oplus \mathfrak{h} \oplus
  \mathfrak{n}^+,\quad \mathfrak{n}^- = \bigoplus_{\beta \in
  \Phi^-} \g_\beta,\quad \g_\beta = \C f_\beta,\quad \beta \in
  \Phi^-.
\]
We fix an ordering $\beta_1, \dots, \beta_m$ of $\Phi^-$.  Then, by
the PBW Theorem, $U(\mathfrak{n}^-)$ has a basis
\[
  \{f_{\beta_1}^{l_1} f_{\beta_2}^{l_2} \dots f_{\beta_m}^{l_m}\ |\
  l_i \in \N,\ 1 \le i \le m\}.
\]
Thus we have a bijection
\[
  \{\text{Isomorphism classes of representations of $Q$}\}
  \stackrel{\text{1-1}}{\longleftrightarrow} \text{Basis of
  $U(\g^-)$}.
\]

Fix a $Q_0$-graded vector space $V = \bigoplus_{i \in Q_0} V_i$ and
let $\bv = \bdim V$.  Then define $A_\bv$ to be the space of
representations of $Q$ on $V$.  More precisely,
\[
  A_\bv := \{ x = (x_a)_{a \in Q_1}\ |\ x_a : V_{s(a)} \to
  V_{t(a)},\ a \in Q_1\}.
\]
Let $G = G_V = \prod_{i \in Q_0} GL(V_i)$.  Then $G$ acts on $A_\bv$
by
\[
  g \cdot x = (g_i)_{i \in Q_0} \cdot (x_a)_{a \in Q_1}
  = (x'_a)_{a \in Q_1}, \quad x'_a =
  g_{t(a)} x_a g_{s(a)}^{-1}.
\]
In this way, the $G$-orbits of $A_\bv$ are precisely the isomorphism
classes of representations of $Q$ with graded dimension $\bv$.
Therefore, we have a bijection
\[
  \{\text{$G$-orbits of $A_\bv$}\}
  \stackrel{\text{1-1}}{\longleftrightarrow} \text{Basis of
  $U(\mathfrak{n}^-)_{-\sum_i v_i \alpha_i}$},
\]
where $U(\mathfrak{n}^-)_{\beta}$ denotes the $\beta$-weight space
of $U(\mathfrak{n}^-)$.

\begin{example} \label{ex:simple-quiver}
Consider the quiver
\smallskip
\begin{center}
\begin{tikzpicture}
  \draw (0,0) node {$Q$};
  \draw (1,0) node {$\bullet$} node[below] {1} to[-<] (2,0) node[below] {$a$};
  \draw (2,0) to (3,0) node {$\bullet$} node[below] {2};
\end{tikzpicture}
\end{center}
Then $\g(Q) = \mathfrak{sl}_3$.  Let $V$ be the $Q_0$-graded vector
space given by $V_1=\C$ and $V_2=\C$ and let $\bv = \bdim V$.  Then
$A_\bv$ is the space of linear maps from $\C$ to $\C$, which is
isomorphic to $\C$ (since any such map is given by multiplication by
some complex number).  Then $A_\bv \cong \C$ decomposes as a
disjoint union of two orbits: $\C^*$ and $\{0\}$.

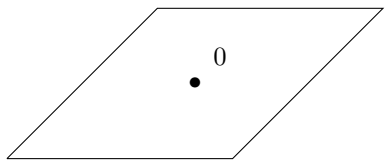
\medskip
\begin{center}
\begin{tikzpicture}
  \draw (0,0) to (2,2) to (5,2) to (3,0) to (0,0);
  \draw (2.5,1) node {$\bullet$} node[label=45:0]{};
\end{tikzpicture}
\end{center}
\end{example}

We would like to construct a variety whose irreducible components
are in one-to-one correspondence with the elements of a basis of a
weight space of $U(\mathfrak{n}^-)$.  The variety $A_\bv$ does not
satisfy this property because it has only one irreducible component.
The problem is that its orbits fit together in such a way as to form
a single irreducible component.  We would like each orbit to give
rise to its own irreducible component.  One way of doing this would
be to take the union of the conormal bundles to the orbits in
$A_\bv$.

We recall the definition of the conormal bundle.  The conormal
bundle to a smooth subvariety $S$ of a smooth variety $A$ is the
sub-bundle of the cotangent bundle $T^*A$ whose fiber over any point
$x \in S$ consists of those $\phi \in (T^*A)_x$ such that
$\phi(v)=0$ for all $v \in (TS)_x$ and whose fiber over all other
points is empty. The dimension of the conormal bundle is equal to
the dimension of $A$ (and half the dimension of the cotangent bundle
$T^*A$).  The larger the dimension of the orbit, the smaller the
dimension of the fibers.

We thus see that the union of the conormal bundles to the orbits in
$A_\bv$ has the property we desire.  Namely, its set of irreducible
components is in one-to-one correspondence with a basis of the
$-\sum_i v_i \alpha_i$ weight space of $U(\mathfrak{n}^-)$.  In
particular, the conormal bundle to the orbit $\C^*$ is simply the
orbit itself while the conormal bundle to the orbit $\{0\}$ is a
complex line.  We will see in the next section that this observation
leads naturally to a ``doubling'' of the quiver and the definition
of the Lusztig quiver variety.

%%%%%%%%%%%%%%%%%%%%%%%%%%%%%%%%%%%%%%%%%%%%%%%%%%%%%%%%%%%%%%%%%%%%
%
\section{The Lusztig quiver variety} \label{S:sec:LQV}
%
%%%%%%%%%%%%%%%%%%%%%%%%%%%%%%%%%%%%%%%%%%%%%%%%%%%%%%%%%%%%%%%%%%%%

Let $\g$ be a Kac-Moody algebra with symmetric Cartan matrix.  For
example, $\g$ could be a simple Lie algebra of type $A$, $D$ or $E$.
We let $Q=(Q_0,Q_1)$ be the \emph{double quiver} associated to the
Dynkin graph of $\g$.  That is, $Q_0$ is the set of vertices of this
Dynkin graph and for each (undirected) edge of the Dynkin graph,
$Q_1$ contains two arrows (one in each direction) with the same
endpoints.  For example, if $\g=\mathfrak{sl}_{n+1}$ is the Lie
algebra of type $A_n$, then the corresponding double quiver is as
follows.
\bigskip
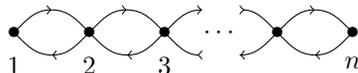
\begin{center}
\begin{tikzpicture}
%  [inner sep=0pt, minimum size=1.5mm]
  \draw[->] node[label=270:1] {$\bullet$} (0,0) to[out=45,in=180]
  (0.5,0.3);
  \draw (0.5,0.3) to[out=0,in=135] (1,0) node[label=270:2] {$\bullet$};
  \draw[->] (1,0) to[out=225,in=0] (0.5,-0.3);
  \draw (0.5,-0.3) to[out=180,in=-45] (0,0);
  \draw[->] (1,0) to[out=45,in=180] (1.5,0.3);
  \draw (1.5,0.3) to[out=0,in=135] (2,0) node[label=270:3] {$\bullet$};
  \draw[->] (2,0) to[out=225,in=0] (1.5,-0.3);
  \draw (1.5,-0.3) to[out=180,in=-45] (1,0);
  \draw[->] (2,0) to[out=45,in=180] (2.5,0.3);
  \draw[>-] (2.5,-0.3) to[out=180,in=-45] (2,0);
  \draw (2.75,0) node {$\cdots$};
  \draw[>-] (3,0.3) to[out=0,in=135] (3.5,0) node {$\bullet$};
  \draw[->] (3.5,0) to[out=225,in=0] (3,-0.3);
  \draw[->] (3.5,0) to[out=45,in=180] (4,0.3);
  \draw (4,0.3) to[out=0,in=135] (4.5,0) node[label=270:$n$] {$\bullet$};
  \draw[->] (4.5,0) to[out=225,in=0] (4,-0.3);
  \draw (4,-0.3) to[out=180,in=-45] (3.5,0);
\end{tikzpicture}
\end{center}
\bigskip

We have a natural involution $a \mapsto \bar a$, $a \in Q_1$, which
maps the arrow $a$ to the arrow with the same underlying edge but
with opposite orientation.

An \emph{orientation} of $Q$ is a choice of subset $\Omega \subseteq
Q_1$ such that $\Omega \cup \bar \Omega = Q_1$ and $\Omega \cap \bar
\Omega = \emptyset$.  That is, $\Omega$ contains exactly one arrow
from each pair associated to each edge of the Dynkin graph.

For a $Q_0$-graded vector space $V=\bigoplus_{i \in Q_0} V_i$, let
\[
  E_V = \bigoplus_{a \in Q_1} \Hom (V_{s(a)}, V_{t(a)}).
\]
Then $E_V = E_{V, \Omega} \oplus E_{V, \bar \Omega}$, where
\[
  E_{V,\Omega} = \bigoplus_{a \in \Omega} \Hom (V_{s(a)},
  V_{t(a)}),\quad E_{V,\bar \Omega} = \bigoplus_{a \in \bar \Omega}
  \Hom (V_{s(a)}, V_{t(a)}).
\]
Recall that $G_V = \prod_{i \in Q_0} GL(V_i)$ acts naturally on
$E_V$, $E_{V,\Omega}$ and $E_{V,\bar \Omega}$.

Define
\[
  \epsilon : Q_1 \to \{\pm 1\},\quad \epsilon(a) =
  \begin{cases}
    +1 & a \in \Omega, \\
    -1 & a \in \bar \Omega.
  \end{cases}
\]
We then define a $G_V$-invariant, nondegenerate, symplectic form
$\langle \cdot, \cdot \rangle$ on $E_V$ by
\[
  \langle x,y \rangle = \sum_{a \in Q_1} \epsilon(a) \tr(x_a
  x_{\bar a}).
\]
This pairs $E_{V,\Omega}$ with $E_{V,\bar \Omega}$ and so we can
view $E_{V,\bar \Omega}$ as the dual space $(E_{V,\Omega})^*$ and
$E_V$ as the tangent space $T^* E_{V,\Omega}$ to $E_{V,\Omega}$.
This is simply the observation that $\Hom(V,W)$ is dual to
$\Hom(W,V)$ under the trace.

To any Hamiltonian action of a Lie group on a symplectic manifold,
there is an associated \emph{moment map}.  The $G_V$-action on $E_V$
is such an action and the corresponding moment map is
\[
  \psi : E_V \to \mathfrak{gl}_V = \prod_{i \in Q_0}
  \mathfrak{gl}(V_i) = \prod_{i \in Q_0} \End V_i,
\]
with $i$th component
\[
  \psi_i(x) = \sum_{a \in Q_1,\, t(a)=i} \epsilon(a) x_a x_{\bar
  a}.
\]
Here $\mathfrak{gl}_V$ is the Lie algebra of $G_V$.  Usually the
moment map is a map to the dual of the Lie algebra of the group but
we have identified $\mathfrak{gl}_V$ with its dual via the trace
here.

We say that $x \in E_V$ is \emph{nilpotent} if there exists an $N
\ge 1$ such that for any path $\beta = a_N \cdots a_2 a_1$ of length
$N$, we have that
\[
  x_{a_N} \cdots x_{a_2} x_{a_1} : V_{s(a_1)} \to V_{t(a_N)}
\]
is the zero map.

\begin{definition}[Lusztig quiver variety] \label{S:def:LQV}
For a $Q_0$-graded vector space $V=\bigoplus_{i \in Q_0} V_i$, let
\[
  \Lambda_V := \{x \in E_V\ |\ \psi(x) = 0,\ x \text{ is
  nilpotent}\}.
\]
The variety $\Lambda_V$ is called the \emph{Lusztig quiver variety}
associated to $Q$ and $V$.
\end{definition}

For $\g$ (or $Q$) of arbitrary type, $\Lambda_V$ has the following
properties:
\begin{enumerate}
  \item $\Lambda_V$ is a closed subvariety of $E_V$ of pure
  dimension $\frac{1}{2} \dim E_V$.  That is, each irreducible
  component of $\Lambda_V$ has this dimension.

  \item $\Lambda_V$ is a langrangian subvariety of $E_V$.

  \item If $x_\Omega \in E_{V, \Omega}$ and $x_{\bar \Omega} \in
  E_{V, \bar \Omega}$, then
  \begin{center}
  \begin{tabular}{rcl}
    $\psi(x_\Omega + x_{\bar \Omega})=0$ & $\iff$ & \parbox{7.5cm}{$x_{\bar \Omega}$
    is orthogonal to the tangent space to the $G_V$-orbit through
    $x_\Omega$ (with respect to the form $\langle \cdot, \cdot
    \rangle$).}
  \end{tabular}
  \end{center}
\end{enumerate}

\smallskip

Additionally, if $\g$ is of finite type, then
\begin{enumerate} \setcounter{enumi}{3}
  \item If $x \in E_V$, then $\psi(x)=0$ implies that $x$ is
  nilpotent.  Thus, the nilpotency condition in the definition of
  the Lusztig quiver variety is superfluous.

  \item The irreducible components of $\Lambda_V$ are the closures
  of the conormal bundles of the $G_V$-orbits in $E_{V,\Omega}$.
\end{enumerate}

Our goal is to use Lusztig quiver varieties to construct the crystal
$B(\infty)$.  We want the following relationship between the
geometry and the elements of the crystal.
\bigskip
\begin{center}
\begin{tabular}{cc}
  \toprule
  Crystal & Geometry \\
  \midrule
  Vertex set & Irreducible components of $\bigsqcup_{V} \Lambda_V$ \\
  Vertices of weight $-\sum v_i \alpha_i$ & Irreducible components
  of $\Lambda_V$, $\bdim V = \bv$ \\
  Crystal operators & Natural geometrically defined operators \\
  \bottomrule
\end{tabular}
\end{center}
\bigskip
In the union $\bigsqcup_V \Lambda_V$, we take the union over one
$Q_0$-graded vector space $V$ of each graded dimension.

We now describe this process in more detail.  For $\bv \in
\N^{Q_0}$, set $V^\bv = \bigoplus_{i \in Q_0} \C^{v_i}$ and
$\Lambda(\bv) = \Lambda_{V^\bv}$.  For $i \in Q_0$, let $\be^i \in
\N^{Q_0}$ such that $\be^i_j = \delta_{ij}$ and for $c \in \N$,
define $\tilde \Lambda(\bv,c\be^i)$ to be the variety of triples
$(x,\phi',\bar \phi)$ where $x \in \Lambda(\bv)$ and
$\phi'=(\phi_i')_{i \in I}$, $\bar \phi = (\bar \phi_i)_{i \in I}$
give an exact sequence
\[
  0 \to V^{\bv-c\be^i} \xrightarrow{\phi'} V^\bv \xrightarrow{\bar
  \phi} V^{c\be^i} \to 0
\]
such that $\im \phi'$ is $x$-invariant.  Then $x$ induces $x' \in
\Lambda(\bv-c\be^i)$ by the restriction to $\im \phi'$ and $\bar x
\in \Lambda(c\be^i)=\{0\}$ by passing to the quotient $V^\bv/\im
\phi'$. Note that $x$ is nilpotent if and only if $x'$ is.

Consider the maps
\[
  \Lambda(\bv-c\be^i) \xleftarrow{p_1}
  \tilde \Lambda(\bv,c\be^i) \xrightarrow{p_2} \Lambda(\bv)
\]
where $p_1(x,\phi',\bar \phi) = x'$ and $p_2(x,\phi',\bar \phi)=x$.
We want to use these maps to identify irreducible components.  The
problem is that they are not the right type of maps (that is, fiber
bundles with smooth fibers).  Thus, we need to restrict them.

For $i \in Q_0$, define $\varepsilon_i : \Lambda(\bv) \to \N$ by
\[
  \varepsilon_i(x) = \dim \Coker \left( \bigoplus_{a,\, t(a)=i}
  V_{s(a)} \xrightarrow{(x_a)} V_i \right).
\]
We will see that this map will play the role of the map
$\varepsilon_i$ in the definition of crystals (see
\lecturecite{Section~\ref{K:section_abstract}}{\cite[Section~5]{Kang}}),
hence the notation. For $c \in \N$, define
\[
  \Lambda(\bv)_{i,c} = \{x \in \Lambda(\bv)\ |\ \varepsilon_i(x) =
  c\}.
\]
This is a locally closed subvariety of $\Lambda(\bv)$.

If $\Lambda(\bv)_{i,c} \ne \emptyset$, then
\[
  p_1^{-1}(\Lambda(\bv-c\be^i)_{i,0}) = p_2^{-1}(\Lambda(\bv)_{i,c})
  := \tilde \Lambda (\bv,c\be^i)_{i,0}
\]
and we have
\begin{equation} \label{S:eq:restricted-projections}
  \Lambda(\bv-c\be^i)_{i,0} \xleftarrow{p_1} \tilde \Lambda(\bv,
  c\be^i)_{i,0} \xrightarrow{p_2} \Lambda(\bv)_{i,c}.
\end{equation}

\begin{lemma}[{\cite[Lemma~5.2.3]{KS97}}]
\label{S:lem:fiber-bundles}
The maps $p_1$ and $p_2$ in \eqref{S:eq:restricted-projections} are
fiber bundles with smooth fibers.
\end{lemma}

\begin{exercise} \label{S:ex:describe-fibers}
Describe the fibers (see \cite[Lemma~5.2.3]{KS97}).
\end{exercise}

\begin{corollary} \label{S:cor:irred-comp-bijection}
If $\Lambda(\bv)_{i,c} \ne \emptyset$, then we have a bijective
correspondence
\begin{center}
\begin{tabular}{rcl}
  Irreducible components of $\Lambda(\bv-c\be^i)_{i,0}$ &
  $\stackrel{\text{1-1}}{\longleftrightarrow}$ & Irreducible components
  of $\Lambda(\bv)_{i,c}$.
\end{tabular}
\end{center}
\end{corollary}

Let $B(\bv,\infty)$ be the set of irreducible components of
$\Lambda(\bv)$ and $B_g(\infty) = \bigsqcup_\bv B(\bv,\infty)$.
Since the sets $\Lambda(\bv)_{i,c}$ are locally closed, for each $X
\in B(\bv,\infty)$ there is an open dense subset of $X$ where
$\varepsilon_i$ takes a fixed value.  We define $\varepsilon_i(X)$
to be this value.  For $c \in \N$, we let
\[
  B(\bv,\infty)_{i,c} = \{X \in B(\bv,\infty)\ |\
  \varepsilon_i(X)=c\}.
\]
By Corollary~\ref{S:cor:irred-comp-bijection}, we have a bijective
correspondence
\[
  B(\bv-c\be^i)_{i,0} \cong B(\bv)_{i,c},\quad \bar X
  \leftrightarrow X.
\]

Define maps
\begin{align*}
  \tilde f_i^c &: B(\bv-c\be^i,\infty)_{i,0} \to
  B(\bv,\infty)_{i,c},\quad \tilde f_i^c(\bar X) = X, \\
  \tilde e_i^c &: B(\bv,\infty)_{i,c} \to B(\bv-c\be^i,
  \infty)_{i,0},\quad \tilde e_i^c(X) = \bar X,
\end{align*}
and then define
\[
  \tilde e_i, \tilde f_i : B_g(\infty) \to B_g(\infty) \sqcup \{0\}
\]
by
\begin{align*}
  \tilde e_i^c &: B(\bv,\infty)_{i,c} \xrightarrow{\tilde e_i^c}
  B(\bv-c\be^i,\infty)_{i,0}
  \xrightarrow{\tilde f_i^{c-1}} B(\bv-\be^i, \infty)_{i,c-1},\quad c >0,\\
  \tilde f_i^c &: B(\bv,\infty)_{i,c} \xrightarrow{\tilde e_i^c}
  B(\bv-c\be^i,\infty)_{i,0} \xrightarrow{\tilde f_i^{c+1}}
  B(\bv+\be^i,\infty)_{i,c+1}.
\end{align*}
We set $\tilde e_i(X)=0$ for $X \in B(\bv,\infty)_{i,0}$.  We also
define
\begin{gather*}
  \wt : B_g(\infty) \to P,\quad \wt(X) = -\sum_{i \in Q_0} v_i
  \alpha_i \text{ for } X \in B(\bv,\infty),\\
  \varphi_i(X) = \varepsilon_i(X) + \langle h_i, \wt(X) \rangle.
\end{gather*}

Here $h_i=[e_i,f_i]$ is the usual element of the Cartan subalgebra
of $\g$ and $\langle \cdot, \cdot \rangle$ denotes the pairing of
$\mathfrak{h}$ with $\mathfrak{h}^*$.  The following theorem was
proved by Kashiwara and Saito.

\begin{theorem}[{\cite[Theorem~5.3.2]{KS97}}] \label{S:thm:Binf-crystal}
The definitions above endow $B_g(\infty)$ with the structure of a
$\g$-crystal and $B_g(\infty)$ is isomorphic to $B(\infty)$, the
crystal corresponding to the lower half of the quantized enveloping
algebra, $U_q^-(\g)$.
\end{theorem}

%%%%%%%%%%%%%%%%%%%%%%%%%%%%%%%%%%%%%%%%%%%%%%%%%%%%%%%%%%%%%%%%%%%%
%
\section{The lagrangian Nakajima quiver variety} \label{S:sec:Nak-QV}
%
%%%%%%%%%%%%%%%%%%%%%%%%%%%%%%%%%%%%%%%%%%%%%%%%%%%%%%%%%%%%%%%%%%%%

The goal in this section is to modify the definition of the Lusztig
quiver variety to obtain varieties giving a geometric realization of
the crystals of irreducible integrable highest weight
representations. To motivate the definitions, we first recall the
construction of these representations via Verma modules.

As before, we let $\g$ be a Kac-Moody algebra with symmetric Cartan
matrix, consider a triangular decomposition
\[
  \g = \mathfrak{n}^- \oplus \mathfrak{h} \oplus \mathfrak{n}^+,
\]
and set $\mathfrak{b}^\pm = \mathfrak{h} \oplus \mathfrak{n}^\pm$.
We fix a highest weight $\lambda \in P^+$, where $P^+$ is the
dominant weight lattice.  Then we define a
$U(\mathfrak{b}^+)$-module $\C_\lambda = \C v_\lambda$ by
\begin{align*}
  h \cdot v_\lambda &= \lambda(h) v_\lambda,\quad h \in \mathfrak{h}, \\
  x \cdot v_\lambda &= 0,\quad x \in \mathfrak{n}^+.
\end{align*}
The \emph{Verma module} of highest weight $\lambda$ is then defined
to be the $U(\g)$-module
\[
  M(\lambda) := U(\g) \otimes_{U(\mathfrak{b}^+)} \C_\lambda,
\]
with the $U(\g)$-action given by left multiplication on the left
factor. It follows that $M(\lambda) \cong U(\mathfrak{n}^-)
\otimes_\C \C_\lambda$ as vector spaces.  There exists a unique
maximal submodule $I(\lambda)$ of $M(\lambda)$ and the irreducible
integrable highest weight representation of highest weight $\lambda$
is defined to be the quotient $V(\lambda) := M(\lambda)/I(\lambda)$.

We summarize this construction as follows:
\begin{equation} \label{S:verma-construction}
  \xymatrix{
    U(\mathfrak{n}^-) \ar@{~>}[rrr]^{\text{shift weights}} & & &
    M(\lambda) \ar@{~>}[rr]^{\text{cut}} & & V(\lambda).
  }
\end{equation}

An analogous phenomenon occurs for crystals.  Namely, the crystal
$B(\lambda)$ embeds into the crystal $B(\infty)$ in the following
manner.  For all $\lambda \in P^+$, there exists a map $\psi_\lambda
: B(\lambda) \to B(\infty)$ such that
\begin{enumerate}
  \item $\psi_\lambda$ is injective,
  \item $\psi_\lambda(b_\lambda) = 1$,
  \item $\psi_\lambda (\tilde f_i b) = \tilde f_i \psi_\lambda(b)$
  when $\tilde f_i(b)=0$,
  \item $\psi_\lambda (\tilde e_i b) = \tilde e_i \psi_\lambda(b)$
  for all $b \in B(\lambda)$, and
  \item \label{S:crystal-embedding-weight-shift}
  $\wt \psi_\lambda(b) = \wt b - \lambda$,
  $\varepsilon_i(\psi_\lambda(b)) = \varepsilon_i(b)$ for all $b \in
  B(\lambda)$.
\end{enumerate}
Property \eqref{S:crystal-embedding-weight-shift} is analogous to
the ``shifting of weights'' in \eqref{S:verma-construction} and the
fact that some $\tilde f_i$'s act as zero in $B(\lambda)$ (while
this does not occur in $B(\infty)$) corresponds to the ``cutting''
in \eqref{S:verma-construction}.  We now aim to mimic these
procedures using quiver varieties.

Let $Q=(Q_0,Q_1)$ be the double quiver corresponding to our
Kac-Moody algebra $\g$ and choose an orientation $\Omega$.  Fix
$\bv, \bw \in \N^{Q_0}$ and let $V$ and $W$ be $Q_0$-graded vector
spaces of graded dimensions $\bv$ and $\bw$ respectively.  The
graded dimension $\bw$ will correspond to the highest weight
\begin{equation} \label{S:eq:omega-w} \textstyle
  \omega_w := \sum_{i \in Q_0} w_i \omega_i
\end{equation}
of the representation whose crystal we would like to construct (here
$\omega_i$ is the $i$th fundamental weight of $\g$). The graded
dimension $\bv$ corresponds to a weight space in that
representation.  More specifically, it corresponds to the weight
space of weight
\begin{equation} \label{S:eq:alpha-v} \textstyle
  \omega_w - \alpha_v, \text{ where } \alpha_v = \sum_{i \in Q_0}
  v_i \alpha_i.
\end{equation}
Here $\alpha_i$ is the $i$th simple root of $\g$.

Define
\[ \textstyle
  \Lambda(\bv,\bw) := \Lambda(\bv) \times \bigoplus_{i \in Q_0}
  \Hom(V_i,W_i).
\]
See Figure~\ref{S:fig:Lambda-v-w}. The space $\Hom(V_i,W_i)$ is
affine and hence we have a bijection
\[
  \text{Irreducible components of } \Lambda(\bv,\bw)
  \stackrel{\text{1-1}}{\longleftrightarrow} \text{Irreducible components
  of $\Lambda(\bv)$}.
\]
We think of the passage from $\Lambda(\bv)$ to $\Lambda(\bv,\bw)$ as
the analogue of the ``shifting weights'' procedure in
\eqref{S:verma-construction}.

\begin{figure}
\begin{center}
\begin{tikzpicture}
% Quiver with V's
  \draw (0,0) node[label=270:{$V_1$}] {$\bullet$} to[->,out=45,in=180] (0.5,0.3);
  \draw (0.5,0.3) to[out=0,in=135] (1,0);
  \draw (0,0) to[-<,out=-45,in=180] (0.5,-0.3);
  \draw (0.5,-0.3) to[out=0,in=225] (1,0);
$
  \draw (1,0) node[label=270:{$V_2$}] {$\bullet$} to[->,out=45,in=180] (1.5,0.3);
  \draw (1.5,0.3) to[out=0,in=135] (2,0);
  \draw (1,0) to[-<,out=-45,in=180] (1.5,-0.3);
  \draw (1.5,-0.3) to[out=0,in=225] (2,0);
$
  \draw (2,0) node[label=270:{$V_3$}] {$\bullet$} to[->,out=45,in=180] (2.5,0.3);
  \draw (2.5,0.3) to[out=0,in=135] (3,0);
  \draw (2,0) to[-<,out=-45,in=180] (2.5,-0.3);
  \draw (2.5,-0.3) to[out=0,in=225] (3,0);
$
  \draw (3,0) node[label=270:{$V_4$}] {$\bullet$} to[->,out=45,in=180] (3.5,0.3);
  \draw (3.5,0.3) to[out=0,in=135] (4,0);
  \draw (3,0) to[-<,out=-45,in=180] (3.5,-0.3);
  \draw (3.5,-0.3) to[out=0,in=225] (4,0);
$
  \draw (4,0) node[label=270:{$V_5$}] {$\bullet$} to[->,out=45,in=180] (4.5,0.3);
  \draw (4.5,0.3) to[out=0,in=135] (5,0);
  \draw (4,0) to[-<,out=-45,in=180] (4.5,-0.3);
  \draw (4.5,-0.3) to[out=0,in=225] (5,0);
  \draw (5,0) node[label=270:{$V_6$}] {$\bullet$};
  \draw[rotate around={30.96:(5,0)}] (5,0) to[->,out=45,in=180] (5.61,0.2);
  \draw[rotate around={30.96:(5,0)}] (5.61,0.2) to[out=0,in=135] (6.22,0)
    node {$\bullet$} node[below right] {$V_7$};
  \draw[rotate around={30.96:(5,0)}] (5,0) to[-<,out=-45,in=180] (5.61,-0.2);
  \draw[rotate around={30.96:(5,0)}] (5.61,-0.2) to[out=0,in=-135] (6.22,0);
  \draw[rotate around={-30.96:(5,0)}] (5,0) to[->,out=45,in=180] (5.61,0.2);
  \draw[rotate around={-30.96:(5,0)}] (5.61,0.2) to[out=0,in=135] (6.22,0)
    node[label=270:{$V_8$}] {$\bullet$};
  \draw[rotate around={-30.96:(5,0)}] (5,0) to[-<,out=-45,in=180] (5.61,-0.2);
  \draw[rotate around={-30.96:(5,0)}] (5.61,-0.2) to[out=0,in=-135] (6.22,0);
% Extra vertices corresponding to W's
  \draw (0,0) to[->,dashed] (0,1.5);
  \draw (0,1.5) to[dashed] (0,3) node {$\bullet$} node[above] {$W_1$};
  \draw (1,0) to[->,dashed] (1,1.5);
  \draw (1,1.5) to[dashed] (1,3) node {$\bullet$} node[above] {$W_2$};
  \draw (2,0) to[->,dashed] (2,1.5);
  \draw (2,1.5) to[dashed] (2,3) node {$\bullet$} node[above] {$W_3$};
  \draw (3,0) to[->,dashed] (3,1.5);
  \draw (3,1.5) to[dashed] (3,3) node {$\bullet$} node[above] {$W_4$};
  \draw (4,0) to[->,dashed] (4,1.5);
  \draw (4,1.5) to[dashed] (4,3) node {$\bullet$} node[above] {$W_5$};
  \draw (5,0) to[->,dashed] (5,1.5);
  \draw (5,1.5) to[dashed] (5,3) node {$\bullet$} node[above] {$W_6$};
  \draw (6,-0.7) to[->,dashed,out=45,in=270] (6.7,0.8);
  \draw (6.7,0.8) to[dashed,out=90,in=-45] (6,2.3) node {$\bullet$} node[above] {$W_8$};
  \draw (6,0.7) to[->,dashed,out=135,in=270] (5.5,2.2);
  \draw (5.5,2.2) to[dashed,out=90,in=225] (6,3.7) node {$\bullet$} node[above] {$W_7$};
\end{tikzpicture}
\end{center}
\caption{The linear maps involved in the definition of the space
$\Lambda(\bv,\bw)$ \label{S:fig:Lambda-v-w} when
$\mathfrak{g}=\mathfrak{so}_{16}$, the simple Lie algebra of type
$D_8$.}
\end{figure}
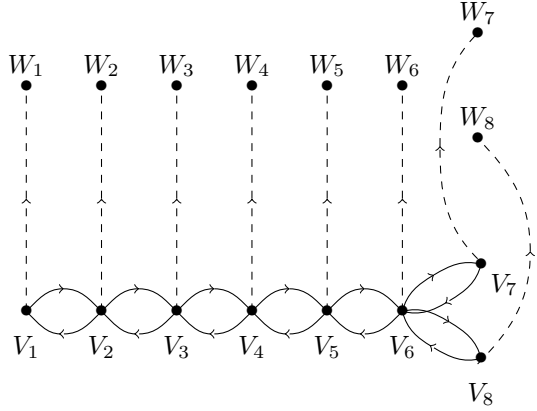

\begin{definition}[Stability condition] \label{S:def:stability}
We say a point $(x,t) \in \Lambda(\bv,\bw)$ is \emph{stable} if the
following condition holds: If $S$ is an $x$-invariant $Q_0$-graded
subspace of $V$ such that $t_i(S_i)=0$ for all $i \in Q_0$, then
$S_i=0$ for all $i \in Q_0$.  We denote the set of stable points in
$\Lambda(\bv,\bw)$ by $\Lambda(\bv,\bw)^\st$.
\end{definition}

The group $G_V$ acts on $\Lambda(\bv,\bw)$ as
\[
  (g,(x,t)) = (x',t'), \text{ where } x_a'=g_{t(a)} x_a
  g_{s(a)}^{-1},\ t_i'=t_i g_i^{-1}.
\]

\begin{lemma}[{\cite[Lemma~3.10]{Nak94}}]
\label{S:lem:trival-stabilizer}
The stabilizer of any point of $\Lambda(\bv,\bw)^\st$ in $G_V$ is
trivial.
\end{lemma}

\begin{definition}[Lagrangian Nakajima quiver variety]
\label{S:def:NQV}
We define
\[
  \mathfrak{L}(\bv,\bw) = \Lambda(\bv,\bw)^\st/G_V,
\]
and we call $\mathfrak{L}(\bv,\bw)$ a \emph{lagrangian Nakajima
quiver variety}.
\end{definition}

\begin{remarks} \label{S:rem:NQV}
\begin{enumerate}
  \item The name lagrangian Nakajima quiver variety arises from the
  fact that the varieties defined here are lagrangian subvarieties
  of what are called (smooth) Nakajima quiver varieties.  See
  \cite{Nak94,Nak98} for details.

  \item For a given irreducible component $X$ of $\Lambda(\bv,\bw)$, one
  of the following two conditions holds:
  \begin{enumerate}
    \item No point of $X$ satisfies the stability condition.
    \item All points in an open dense subset of $X$ satisfy the
    stability condition.
  \end{enumerate}
  Thus we speak of each irreducible component as failing or
  satisfying the stability condition in case the first or second
  condition (respectively) above holds.  Imposing the stability
  condition corresponds to the ``cutting'' procedure in
  \eqref{S:verma-construction}.

  \item Since $G_V$ acts freely on $\Lambda(\bv,\bw)^\st$, there is
  a bijection
  \[
    \text{Irreducible components of } \Lambda(\bv,\bw)^\st
    \stackrel{\text{1-1}}{\longleftrightarrow} \text{Irreducible
    components of } \mathfrak{L}(\bv,\bw).
  \]
\end{enumerate}
\end{remarks}

We summarize the geometric construction as follows (compare to
\eqref{S:verma-construction}):
\[
  \xymatrix{
    \Lambda(\bv) \ar@{~>}[rr]^(0.3){\text{shift weights}} && \Lambda(\bv)
    \times \bigoplus_{i \in Q_0} \Hom(V_i,W_i)
    \ar@{~>}[rr]^(.65){\text{cut}} && \Lambda(\bv,\bw)^\st \ar@{~>}[r]
    & \mathfrak{L}(\bv,\bw).
  }
\]

\begin{example} \label{S:eg:sl2}
Let $\g=\mathfrak{sl}_2$.  In this case $Q$ is the quiver consisting
of a single vertex and no edges.  Then $\bv=v$ and $\bw=w$ are just
natural numbers and
\[
  \Lambda(v,w)^\st = \{t : V \to W\ |\ t \text{ injective}\},
\]
which is empty unless $v \le w$.  If this condition is satisfied,
then
\[
  \mathfrak{L}(v,w) = \Lambda(v,w)^\st/GL(V) \cong \Gr(v,w),
\]
where $\Gr(v,w)$ is the grassmannian of $v$-dimensional subspaces of
$\C^w$ and the isomorphism is induced by the map $t \mapsto \im t$.
\end{example}

\begin{exercise} \label{S:ex:flag-variety}
Let $\g=\mathfrak{sl}_{n+1}$ and $\bw=N\be^n$ for some $N \in \N$.
Then we can picture the maps involved in the quiver variety as
follows:
\bigskip
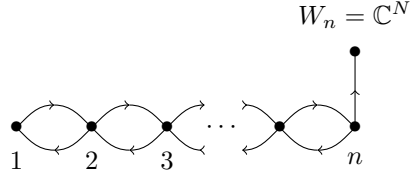
\begin{center}
\begin{tikzpicture}
%  [inner sep=0pt, minimum size=1.5mm]
  \draw[->] node[label=270:1] {$\bullet$} (0,0) to[out=45,in=180]
  (0.5,0.3);
  \draw (0.5,0.3) to[out=0,in=135] (1,0) node[label=270:2] {$\bullet$};
  \draw[->] (1,0) to[out=225,in=0] (0.5,-0.3);
  \draw (0.5,-0.3) to[out=180,in=-45] (0,0);
  \draw[->] (1,0) to[out=45,in=180] (1.5,0.3);
  \draw (1.5,0.3) to[out=0,in=135] (2,0) node[label=270:3] {$\bullet$};
  \draw[->] (2,0) to[out=225,in=0] (1.5,-0.3);
  \draw (1.5,-0.3) to[out=180,in=-45] (1,0);
  \draw[->] (2,0) to[out=45,in=180] (2.5,0.3);
  \draw[>-] (2.5,-0.3) to[out=180,in=-45] (2,0);
  \draw (2.75,0) node {$\cdots$};
  \draw[>-] (3,0.3) to[out=0,in=135] (3.5,0) node {$\bullet$};
  \draw[->] (3.5,0) to[out=225,in=0] (3,-0.3);
  \draw[->] (3.5,0) to[out=45,in=180] (4,0.3);
  \draw (4,0.3) to[out=0,in=135] (4.5,0) node[label=270:$n$] {$\bullet$};
  \draw[->] (4.5,0) to[out=225,in=0] (4,-0.3);
  \draw (4,-0.3) to[out=180,in=-45] (3.5,0);
  \draw[->] (4.5,0) to (4.5,0.5);
  \draw (4.5,0.5) to (4.5,1) node[label=90:{$W_n=\C^N$}] {$\bullet$};
\end{tikzpicture}
\end{center}
\bigskip
Show that $\mathfrak{L}(v,w)$ is empty unless $v_1 \le v_2 \le \dots
\le v_n \le N$, in which case
\begin{align*}
  \mathfrak{L}(\bv,\bw) &\cong \{0 \subseteq V_1 \subseteq \dots
  \subseteq V_n \subseteq \C^N\ |\ \dim V_i = v_i\}, \\
  (x,t) &\mapsto (V_i), \text{ where } V_i = \im t a_{n-1\to n}
  a_{n-2 \to n-1} \dots a_{i \to i+1}.
\end{align*}
Here $a_{j \to j+1}$ is the arrow from vertex $j$ to vertex $j+1$
for $1 \le j \le n-1$.  Thus $\mathfrak{L}(\bv,\bw)$ is the partial
flag variety. See
\lecturecite{Section~\ref{J:se15}}{\cite[Section~4.1]{Kamnitzer}}.
\end{exercise}

\begin{example} \label{S:eg:NQV-sl3-adjoint}
Let $\g=\mathfrak{sl}_3$, $\bw = \be^1 + \be^2$, and
$\bv=\be^1+\be^2$.  This corresponds, under the identifications
\eqref{S:eq:omega-w} and \eqref{S:eq:alpha-v}, to the weight zero
subspace of the adjoint representation of $\mathfrak{sl}_3$ (which
is two-dimensional).  We can picture the maps involved in our quiver
varieties as follows:
\bigskip
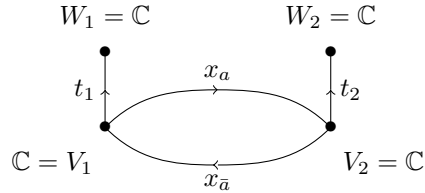
\begin{center}
\begin{tikzpicture}
  \draw[->] node[label=260:{$\C=V_1$}] {$\bullet$} (0,0)
  to[out=45,in=180] (1.5,0.5) node[above] {$x_a$};
  \draw (1.5,0.5) to[out=0,in=135] (3,0) node[label=280:{$V_2=\C$}]
  {$\bullet$};
  \draw[-<] (0,0) to[out=-45,in=180] (1.5,-0.5) node[below] {$x_{\bar a}$};
  \draw (1.5,-0.5) to[out=0,in=225] (3,0);
  \draw (0,0) to[->] (0,0.5) node[left] {$t_1$};
  \draw (0,0.5) to (0,1) node[label=90:{$W_1=\C$}] {$\bullet$};
  \draw (3,0) to[->] (3,0.5) node[right] {$t_2$};
  \draw (3,0.5) to (3,1) node[label=90:{$W_2=\C$}] {$\bullet$};
\end{tikzpicture}
\end{center}
\bigskip
We split the points of the quiver variety into two cases. We first
consider the case where $x_a=0$. Then the stability condition
implies that
\[
  \ker x_{\bar a} \cap \ker t_2 \ne 0,\quad t_1 \ne 0.
\]
There is a unique element of $GL(V_1)$ whose action changes $t_1$ to
the identity map.  Therefore, we can describe the quotient defining
the subvariety of $\mathfrak{L}(\bv,\bw)$ given by $x_a=0$ by fixing
$t_1=1$ and considering only the quotient by $GL(V_2)$.  Thus we can
identify this subvariety of the quiver variety with
\[
  \{(x_{\bar a}, t_2) \in \C^2\ |\ (x_{\bar a}, t_2) \ne (0,0)\}/\C
  \cong \mathbb{P}^1,
\]
where the quotient by $\C$ is the quotient by $GL(V_2)$.  The second
case, where $x_{\bar a}=0$ is analogous and also yields
$\mathbb{P}^1$.  Therefore, the lagrangian Nakajima quiver variety
$\mathfrak{L}(\bv,\bw)$ consists of two projective lines meeting at
a point (the point where $x_a$ and $x_{\bar a}$ are both zero).
\bigskip
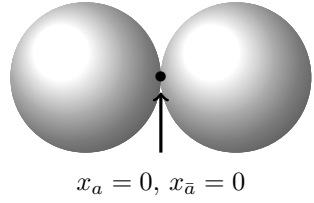
\begin{center}
\begin{tikzpicture}
  \shade[ball color=white] (0,0) circle (1);
  \shade[ball color=white] (2,0) circle (1);
  \draw[->, very thick] (1,-1) node[label=270:{$x_a=0$, $x_{\bar a}=0$}]
  {} to (1,-0.2);
  \draw (1,0) node {$\bullet$};
\end{tikzpicture}
\end{center}
\end{example}

\begin{exercise} \label{S:ex:P1s}
Let $\g=\mathfrak{sl}_{n+1}$, $\bw = \be^1 + \be^n$ and $\bv=\be^1 +
\be^2 + \dots + \be^n$.  Under the identifications
\eqref{S:eq:omega-w} and \eqref{S:eq:alpha-v}, this corresponds to
the $n$-dimensional weight zero subspace of the adjoint
representation of $\mathfrak{sl}_{n+1}$. Show that the lagrangian
Nakajima quiver variety is a union of $n$ projective lines as shown
below.
\bigskip
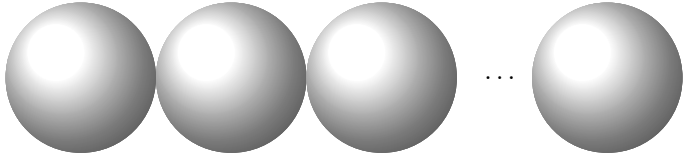
\begin{center}
\begin{tikzpicture}
  \shade[ball color=white] (0,0) circle (1);
  \shade[ball color=white] (2,0) circle (1);
  \shade[ball color=white] (4,0) circle (1);
  \draw (5.6,0) node {$\dots$};
  \shade[ball color=white] (7,0) circle (1);
\end{tikzpicture}
\end{center}
\end{exercise}

Having defined the lagrangian Nakajima quiver varieties, our goal is
to describe the structure of a crystal on the set of irreducible
components of $\bigsqcup_\bv \mathfrak{L}(\bv,\bw)$.  For $c \in
\N$, define
\begin{multline*}
  \mathfrak{F}(\bv,\bw;c\be^i) = \{(x,t,S)\ |\ (x,t) \in
  \Lambda(\bv,\bw)^\st,\, S \text{ an $x$-invariant} \\
  \text{$Q_0$-graded
  subspace of $V$},\, \dim S = \bv-c\be^i\}/GL_V.
\end{multline*}
We then have natural projections
\[
  \mathfrak{L}(\bv-c\be^i,\bw) \xleftarrow{\pi_1} \mathfrak{F}(\bv,
  \bw; c\be^i) \xrightarrow{\pi_2} \mathfrak{L}(\bv,\bw),
\]
where $\pi_1(G_V \cdot (x,t)) = G_S \cdot (x|_S, t|_S)$ and
$\pi_2(G_V \cdot (x,t,S)) = G_V \cdot (x,t)$.  For $c \in \N$,
define
\[ \textstyle
  \varepsilon_i : \mathfrak{L}(\bv,\bw) \to \N,\ \varepsilon_i(x,t)
  = \dim \Coker \left( \bigoplus_{a,\, t(a)=i} V_{s(a)} \to V_i
  \right).
\]
For $i \in Q_0$ and $c \in \N$, let
\[
  \mathfrak{L}(\bv,\bw)_{i,c} = \{G_V \cdot (x,t) \in
  \mathfrak{L}(\bv,\bw)\ |\ \varepsilon_i((x,t)) = c\}.
\]
Then $\mathfrak{L}(\bv,\bw)_{i,c}$ is a locally closed subvariety of
$\mathfrak{L}(\bv,\bw)$.  If $\mathfrak{L}(\bv,\bw)_{i,c} \ne
\emptyset$, then
\[
  \pi_1^{-1}(\mathfrak{L}(\bv-c\be^i,\bw)_{i,0}) =
  \pi_2^{-1}(\mathfrak{L}(\bv,\bw)_{i,c}).
\]
Setting
\[
  \mathfrak{F}(\bv,\bw; c\be^i)_{i,0} =
  \pi_1^{-1}(\mathfrak{L}(\bv-c\be^i,\bw)_{i,0}) =
  \pi_2^{-1}(\mathfrak{L}(\bv,\bw)_{i,c}),
\]
we have the natural projections
\[
  \mathfrak{L}(\bv-c\be^i,\bw)_{i,0} \xleftarrow{\pi_1}
  \mathfrak{F}(\bv,\bw;c\be^i)_{i,0} \xrightarrow{\pi_2}
  \mathfrak{L}(\bv,\bw)_{i,c}.
\]
The restriction of $\pi_2$ to $\mathfrak{F}(\bv,\bw; c\be^i)_{i,0}$
is an isomorphism and $\mathfrak{L}(\bv-c\be^i,\bw)_{i,0}$ is an
open subvariety of $\mathfrak{L}(\bv-c\be^i,\bw)$.

\begin{lemma}[{\cite[Lemma~4.2.2]{S02}}] \label{S:lem:fibers}
\begin{enumerate}
  \item For $i \in Q_0$,
  \[
    \mathfrak{L}(\mathbf{0}, \bw)_{i,c} = \begin{cases}
      \mathrm{pt} & \text{if } c=0, \\
      \emptyset & \text{if } c>0.
    \end{cases}
  \]

  \item If $\mathfrak{L}(\bv,\bw)_{i,c} \ne \emptyset$, then the
  fiber of the restriction of $\pi_1$ to $\mathfrak{F}(\bv,\bw;
  c\be^i)_{i,0}$ is isomorphic to a grassmannian variety.
\end{enumerate}
\end{lemma}

\begin{corollary} \label{S:cor:NQV-irrcomp-corr}
If $\mathfrak{L}(\bv,\bw)_{i,c} \ne \emptyset$, then there is a
one-to-one correspondence between the set of irreducible components
of $\mathfrak{L}(\bv-c\be^i,\bw)_{i,0}$ and the set of irreducible
components of $\mathfrak{L}(\bv,\bw)_{i,c}$.
\end{corollary}

Let $B(\bv,\bw)$ be the set of irreducible components of
$\mathfrak{L}(\bv,\bw)$ and $B_g(\bw) = \bigsqcup_\bv B(\bv,\bw)$.
As in Section~\ref{S:sec:LQV}, for $X \in B(\bv,\bw)$, let
$\varepsilon_i(X)= \varepsilon_i((x,t))$ for a generic point $G_V
\cdot (x,t) \in X$. Then for $c \in \N$, define
\[
  B(\bv,\bw)_{i,c} = \{X \in B(\bv,\bw)\ |\ \varepsilon_i(X) = c\}.
\]
Thus, by Corollary~\ref{S:cor:NQV-irrcomp-corr} we have a bijective
correspondence
\[
  B(\bv-c\be^i, \bw)_{i,0} \cong B(\bv,\bw)_{i,c},\quad \bar X
  \leftrightarrow X.
\]
We then define maps
\begin{align*}
  \tilde f^c_i &: B(\bv-c\be^i,\bw)_{i,0} \to B(\bv,\bw)_{i,c},\quad
  \tilde f^c_i(\bar X)=X, \\
  \tilde e^c_i &: B(\bv,\bw)_{i,c} \to B(\bv-c\be^i,\bw)_{i,0},\quad
  \tilde e^c_i (X) = \bar X.
\end{align*}
Then we define
\[
  \tilde e_i,\ \tilde f_i : B(\bw) \to B(\bw) \sqcup \{0\}
\]
by
\begin{align*}
  \tilde e_i &: B(\bv,\bw)_{i,c} \xrightarrow{\tilde e^c_i} B(\bv -
  c\be^i, \bw)_{i,0} \xrightarrow{\tilde f^{c-1}_i} B(\bv-\be^i,\bw)_{i,c-1}, \\
  \tilde f_i &: B(\bv,\bw)_{i,c} \xrightarrow{\tilde e^c_i} B(\bv -
  c\be^i, \bw)_{i,0} \xrightarrow{\tilde f^{c+1}_i} B(\bv+\be^i,
  \bw)_{i, c+1}.
\end{align*}
We set $\tilde e_i(X) = 0$ for $X \in B(\bv,\bw)_{i,0}$ and $\tilde
f_i(X)=0$ for $X \in B(\bv,\bw)_{i,c}$ with $B(\bv,\bw)_{i,c+1} =
\emptyset$.  Furthermore, we define
\begin{gather*}
  \wt : B_g(\bw) \to P,\quad \wt X = \omega_\bw - \alpha_\bv \text{
  for } X \in B(\bv,\bw), \\
  \varphi_i(X) = \varepsilon_i(X) + \langle h_i, \wt X \rangle.
\end{gather*}
The following theorem was proven by Saito.

\begin{theorem}[{\cite[Theorem~4.6.4]{S02}}]
\label{S:thm:geometric-hw-crystal}
The definitions above endow $B_g(\bw)$ with the structure of a
$\g$-crystal and $B_g(\bw)$ is isomorphic to the crystal
$B(\omega_\bw)$ of the irreducible integrable highest weight
representation of highest weight $\omega_\bw$.
\end{theorem}

\begin{remarks} \label{S:rem:rep-crystal}
\begin{enumerate}
  \item While Theorem~\ref{S:thm:geometric-hw-crystal} only applies
  to crystals with symmetric Cartan matrix, one can use a ``folding''
  procedure to produce a geometric construction of the crystals of
  highest weight representations in arbitrary symmetrizable type
  (see \cite{Sav04b}).

  \item Nakajima quiver varieties can be used to construct the full
  structure (as opposed to the crystal structure)
  of the irreducible integrable highest weight
  representations of a Kac-Moody algebra $\g$ with symmetric Cartan
  matrix.  The procedure is similar to the one described in
  \lecturecite{Section~\ref{J:se17}}{\cite[Section~4.3]{Kamnitzer}}.  One defines a
  smooth quiver variety $\mathfrak{M}$ and a singular quiver variety
  $\mathfrak{M}_0$ and there exists a resolution of singularities
  $\pi : \mathfrak{M} \to \mathfrak{M}_0$.  Then $\mathfrak{L} =
  \pi^{-1}(0)$ and we define $Z = \mathfrak{M} \times_\pi
  \mathfrak{M}$.  Then convolution gives the homology of $Z$ the
  structure of an algebra (isomorphic to a quotient of $U(\g)$) and
  defines an action of this algebra on the homology of
  $\mathfrak{L}$, realizing the representation in question.  For
  details, see \cite{Nak94,Nak98}.

  \item In type $A$ (i.e. when $\g=\mathfrak{sl}_n$), the quiver
  variety construction and the Ginzburg construction described in
  \lecturecite{Section~\ref{J:se14}}{\cite[Section~4]{Kamnitzer}} are closely related
  but slightly different.  See \cite{Sav04} for details.

  \item One can define quiver varieties which yield geometric
  realizations of tensor products of crystals of highest weight
  representations (see \cite{Mal02,Mal03,Nak01}).
\end{enumerate}
\end{remarks}

%%%%%%%%%%%%%%%%%%%%%%%%%%%%%%%%%%%%%%%%%%%%%%%%%%%%%%%%%%%%%%%%%%%%
%
\section{Connections to combinatorial realizations of crystal
graphs} \label{S:sec:connections}
%
%%%%%%%%%%%%%%%%%%%%%%%%%%%%%%%%%%%%%%%%%%%%%%%%%%%%%%%%%%%%%%%%%%%%

In this section, we will describe the precise relationship between
the combinatorial realizations of crystals given by tableaux and the
geometric realizations introduced in Section~\ref{S:sec:Nak-QV}. For
this section, we fix $\g = \mathfrak{sl}_{n+1}$ to be the simple Lie
algebra of type $A_n$.

Recall that the double quiver $Q$ associated to $\g$ is as follows.
\bigskip
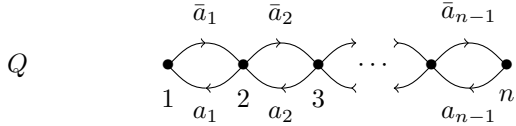
\begin{center}
\begin{tikzpicture}
%  [inner sep=0pt, minimum size=1.5mm]
  \draw (-2,0) node {$Q$};
  \draw[->] node[label=270:1] {$\bullet$} (0,0) to[out=45,in=180]
  (0.5,0.3) node[label=90:{$\bar a_1$}] {};
  \draw (0.5,0.3) to[out=0,in=135] (1,0) node[label=270:2] {$\bullet$};
  \draw[->] (1,0) to[out=225,in=0] (0.5,-0.3) node[label=270:{$a_1$}] {};
  \draw (0.5,-0.3) to[out=180,in=-45] (0,0);
  \draw[->] (1,0) to[out=45,in=180] (1.5,0.3) node[label=90:{$\bar a_2$}] {};
  \draw (1.5,0.3) to[out=0,in=135] (2,0) node[label=270:3] {$\bullet$};
  \draw[->] (2,0) to[out=225,in=0] (1.5,-0.3) node[label=270:{$a_2$}] {};
  \draw (1.5,-0.3) to[out=180,in=-45] (1,0);
  \draw[->] (2,0) to[out=45,in=180] (2.5,0.3);
  \draw[>-] (2.5,-0.3) to[out=180,in=-45] (2,0);
  \draw (2.75,0) node {$\cdots$};
  \draw[>-] (3,0.3) to[out=0,in=135] (3.5,0) node {$\bullet$};
  \draw[->] (3.5,0) to[out=225,in=0] (3,-0.3);
  \draw[->] (3.5,0) to[out=45,in=180] (4,0.3) node[label=90:{$\bar a_{n-1}$}] {};
  \draw (4,0.3) to[out=0,in=135] (4.5,0) node[label=270:$n$] {$\bullet$};
  \draw[->] (4.5,0) to[out=225,in=0] (4,-0.3) node[label=270:{$a_{n-1}$}] {};
  \draw (4,-0.3) to[out=180,in=-45] (3.5,0);
\end{tikzpicture}
\end{center}
\bigskip
We choose the orientation $\Omega=\{a_1, a_2, \dots, a_{n-1}\}$
consisting of the left-pointing arrows and let $Q_\Omega$ be the
corresponding quiver containing only the arrows in $\Omega$:
\bigskip
\begin{center}
\begin{tikzpicture}
  \draw (-2,0) node {$Q_\Omega$};
  \draw[-<] (0,0) node {$\bullet$} node[below] {1} -- (0.5,0);
  \draw (0.5,0) -- (1,0);
  \draw[-<] (1,0) node {$\bullet$} node[below] {2} -- (1.5,0);
  \draw (1.5,0) -- (2,0);
  \draw[-<] (2,0) node {$\bullet$} node[below] {3} -- (2.5,0);
  \draw (2.5,0) -- (3,0) node {$\bullet$} node[below] {4};
  \draw (3,0) -- (3.5,0) node[right] {$\cdots$};
  \draw (4.5,0) -- (5,0);
  \draw[-<] (5,0) node {$\bullet$} -- (5.5,0);
  \draw (5.5,0) -- (6,0);
  \draw[-<] (6,0) node {$\bullet$} -- (6.5,0);
  \draw (6.5,0) -- (7,0) node {$\bullet$} node[below] {$n$};
\end{tikzpicture}
\end{center}
\bigskip

Let $V=\bigoplus_{i \in I} V_i$ be a $Q_0$-graded vector space of
graded dimension $\bdim V = \bv$. By Gabriel's Theorem
(Theorem~\ref{S:thm:Gabriel}) and the discussion in
Section~\ref{S:sec:quivers}, we have the following bijections.
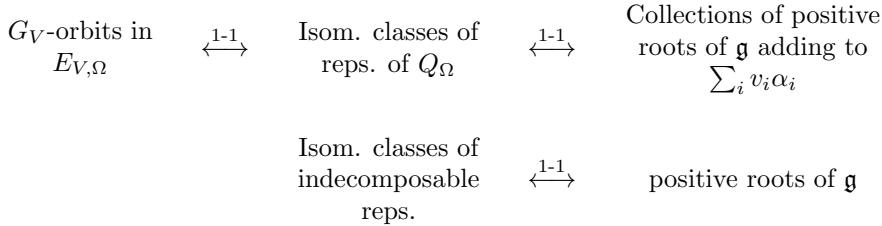
\begin{center}
\begin{tabular}{rcccl}
  \parbox{2.5cm}{\begin{center} $G_V$-orbits in $E_{V,\Omega}$ \end{center}} &
  $\stackrel{\text{1-1}}{\longleftrightarrow}$ & \parbox{3cm}{\begin{center} Isom. classes of reps. of
  $Q_\Omega$ \end{center}} & $\stackrel{\text{1-1}}{\longleftrightarrow}$ & \parbox{4cm}{\begin{center} Collections of
  positive roots of $\g$ adding to $\sum_i v_i \alpha_i$
  \end{center}} \\ \\
  & & \parbox{3cm}{\begin{center} Isom. classes of indecomposable reps. \end{center}} &
  $\stackrel{\text{1-1}}{\longleftrightarrow}$ & \parbox{4cm}{\begin{center}
  positive roots of $\g$ \end{center}}
\end{tabular}
\end{center}
Recall that under these bijections, the positive root $\alpha_i +
\dots + \alpha_j$, $i \le j$, corresponds to the isomorphism class
of the representation $(V^{i,j}, x^{i,j})$ of
Example~\ref{S:eg:An-indecomps}.  We depict this representation by
the diagram
\bigskip
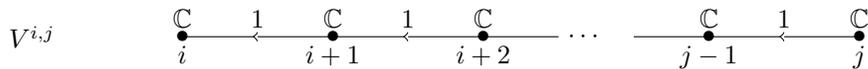
\begin{center}
\begin{tikzpicture}
  \draw (-2,0) node {$V^{i,j}$};
  \draw[-<] (0,0) node {$\bullet$} node[below] {$i$} node[above] {$\C$} -- (1,0) node[above] {1};
  \draw (1,0) -- (2,0);
  \draw[-<] (2,0) node {$\bullet$} node[below] {$i+1$} node[above] {$\C$} -- (3,0) node[above] {1};
  \draw (3,0) -- (4,0);
  \draw[-] (4,0) node {$\bullet$} node[below] {$i+2$} node[above] {$\C$} -- (5,0) node[right] {$\cdots$};
  \draw (6,0) -- (7,0);
  \draw[-<] (7,0) node {$\bullet$} node[below] {$j-1$} node[above] {$\C$} -- (8,0) node[above] {1};
  \draw (8,0) -- (9,0) node {$\bullet$} node[below] {$j$} node[above] {$\C$};
\end{tikzpicture}
\end{center}
\bigskip
or simply by
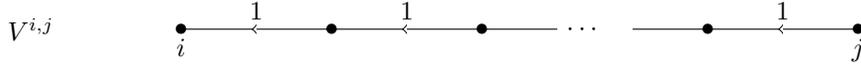
\bigskip
\begin{center}
\begin{tikzpicture}
  \draw (-2,0) node {$V^{i,j}$};
  \draw[-<] (0,0) node {$\bullet$} node[below] {$i$} -- (1,0) node[above] {1};
  \draw (1,0) -- (2,0);
  \draw[-<] (2,0) node {$\bullet$} -- (3,0) node[above] {1};
  \draw (3,0) -- (4,0);
  \draw[-] (4,0) node {$\bullet$} -- (5,0) node[right] {$\cdots$};
  \draw (6,0) -- (7,0);
  \draw[-<] (7,0) node {$\bullet$} -- (8,0) node[above] {1};
  \draw (8,0) -- (9,0) node {$\bullet$} node[below] {$j$};
\end{tikzpicture}.
\end{center}
\bigskip
In such pictures, each vertex represents a basis vector in the
representation, and we identify the two. Vertex labels (elements of
$Q_0$) indicate the degree in which this vector lives and arrow
labels indicate the action of $x$ ($=x^{i,j}$ here). Thus an arrow
labeled by $z \in \C$ from a vertex $a$ to a vertex $b$ indicates
that the coefficient of $b$ in the expansion of $x(a)$ as a linear
combination of the vertices is $z$. We will always vertically line
up vertices of the same degree and degrees will increase from left
to right.

Recall that in finite type, the nilpotency condition in the
definition of the Lusztig quiver varieties is superfluous and thus
\[
  \Lambda(\bv) = \{x \in E_V = E_{V, \Omega} \oplus E_{V, \bar
  \Omega} \ |\ \psi(x) = 0\},
\]
where $\psi(x)=0$ is the moment map condition.  In our case, the
moment map condition is equivalent to the following set of
conditions:
\begin{align}
  x_{a_i} x_{\bar a_i} &= x_{\bar a_{i-1}} x_{a_{i-1}},\quad i \ne
  1,n, \label{S:eq:MMC-An-middle} \\
  x_{a_1} x_{\bar a_1} &= 0, \label{S:eq:MMC-An-1} \\
  x_{\bar a_{n-1}} x_{a_{n-1}} &=0. \label{S:eq:MMC-An-n}
\end{align}
We adopt the convention that $x_{a_i} = 0$ and $x_{\bar a_i}=0$ for
$i \le 0$ or $i \ge n$.

We know that the irreducible components of $\Lambda(\bv)$ are the
closures of the conormal bundles to the orbits in $E_{V,\Omega}$ and
that $E_{V, \bar \Omega}$ corresponds to the cotangent direction.
Suppose we consider an orbit corresponding to a single positive
root. More precisely, we consider the orbit through the point
$(V^{i,j}, x^{i,j})$ for some $i \le j$.  We would like to describe
the conormal bundle to this orbit. Because the group $G_V$ acts
transitively on the orbit and identifies the fibers over points, it
suffices to describe the fiber of the conormal bundle over a
particular point, namely $(V^{i,j}, x^{i,j})$.  We have seen that
this fiber is given by the set
\[
  \{x=(x^{i,j}, x_{\bar \Omega} = (x_{\bar a_k})_{1 \le k \le n-1}) \in E_V\ |\ \psi(x) = 0\}.
\]
Suppose $x=(x^{i,j}, x_{\bar \Omega} = (x_{\bar a_k})_{1 \le k \le
n-1})$ is in the fiber.  Since $V^{i,j}_k=0$ unless $i \le k \le j$,
we have $x_{\bar a_k}=0$ unless $i \le k \le j-1$.  Therefore, by
\eqref{S:eq:MMC-An-middle} (if $i > 1$) or \eqref{S:eq:MMC-An-1} (if
$i=1$), we have $x_{a_i} x_{\bar a_i} = 0$.  But $x_{a_i}=1$ and so
$x_{\bar a_i}=0$.  Again, one can use \eqref{S:eq:MMC-An-middle} to
show that $x_{\bar a_{i+1}}=0$.  Continuing in this manner, we see
that in fact $x_{\bar \Omega}=0$.  Therefore, the conormal bundle to
the orbit is just the orbit itself (i.e. each fiber in the conormal
bundle consists of a single point).

Now consider an orbit corresponding to 2 positive roots. More
precisely, we consider the orbit through the point $x^{i_1,j_1}
\oplus x^{i_2,j_2}$ for some $i_1 \le j_1$ and $i_2 \le j_2$.  We
picture this representation as in
Figure~\ref{S:fig:sum-of-two-reps}.

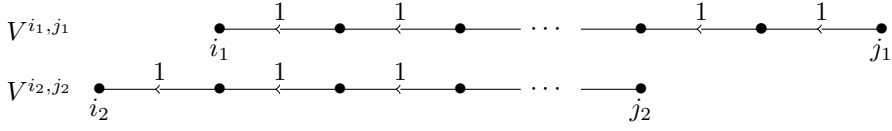
\begin{figure}
\begin{center}
\begin{tikzpicture}[scale=0.8]
% V^{i_1,j_1}
  \draw (-3,0) node {$V^{i_1,j_1}$};
  \draw[-<] (0,0) node {$\bullet$} node[below] {$i_1$} -- (1,0) node[above] {1};
  \draw (1,0) -- (2,0);
  \draw[-<] (2,0) node {$\bullet$} -- (3,0) node[above] {1};
  \draw (3,0) -- (4,0);
  \draw[-] (4,0) node {$\bullet$} -- (5,0) node[right] {$\cdots$};
  \draw (6,0) -- (7,0);
  \draw[-<] (7,0) node {$\bullet$} -- (8,0) node[above] {1};
  \draw (8,0) -- (9,0) node {$\bullet$};
  \draw[-<] (9,0) node {$\bullet$} -- (10,0) node[above] {1};
  \draw (10,0) -- (11,0) node {$\bullet$} node[below] {$j_1$};
  \draw (-3,-1) node {$V^{i_2,j_2}$};
  \draw[-<] (-2,-1) node {$\bullet$} node[below] {$i_2$} -- (-1,-1) node[above] {1};
  \draw (-1,-1) -- (0,-1);
  \draw[-<] (0,-1) node {$\bullet$} -- (1,-1) node[above] {1};
  \draw (1,-1) -- (2,-1);
  \draw[-<] (2,-1) node {$\bullet$} -- (3,-1) node[above] {1};
  \draw (3,-1) -- (4,-1);
  \draw[-] (4,-1) node {$\bullet$} -- (5,-1) node[right] {$\cdots$};
  \draw (6,-1) -- (7,-1) node {$\bullet$} node[below] {$j_2$};
\end{tikzpicture}
\end{center}
\caption{The quiver representation $x^{i_1,j_1} \oplus x^{i_2,j_2}$.
Each vertex represents a basis vector in $V^{i_1,j_1} \oplus
V^{i_2,j_2}$.  Vertex labels indicate the degree of the
corresponding basis vector and vertices of the same degree are
vertically aligned.  Edges indicate the action of $x^{i_1,j_1}
\oplus x^{i_2,j_2}$ with coefficients given by edge labels.
\label{S:fig:sum-of-two-reps}}
\end{figure}

We now wish to describe the conormal bundle to this orbit. Again, it
suffices to describe the fiber of the conormal bundle over a
particular point, namely $(V^{i_1,j_1} \oplus V^{i_2, j_2},
x^{i_1,j_1} \oplus x^{i_2,j_2})$, and this fiber is given by the set
\[
  \{x=(x^{i_1,j_1} \oplus x^{i_2,j_2}, x_{\bar \Omega}) \in E_V\ |\
  \psi(x) = 0\}.
\]
For $l=1,2$, let $V^{i_l,j_l}_k = \C v^l_k$, $i_l \le k \le j_l$,
such that $x^{i,j}_{a_k}(v^l_{k+1}) = v^l_k$ for $i_l \le k \le
j_l-1$. Set $v^l_k=0$ for $k < i_l$ or $k>j_l$.  What are the
possible values of $x_{\bar a_k}(v^2_k)$ for $i_2 \le k \le j_2$? In
general, $x_{\bar a_k}(v^2_k)$ is a linear combination of
$v^1_{k+1}$ and $v^2_{k+1}$. However, by the same argument as above,
one can show that the coefficient of $v^2_{k+1}$ must in fact be
zero.  Therefore
\[
  x_{\bar a_k}(v^2_k) = c_k v^1_{k+1},
\]
for some $c_k \in \C$.  Suppose that $i_1 \le k+1 \le j_1$ and $c_k
\ne 0$.  Then if $k+1 > j_2$, by \eqref{S:eq:MMC-An-middle} we have
\[
  x_{\bar a_{k-1}} (v^2_{k-1}) = x_{\bar a_{k-1}} x_{a_{k-1}}
  (v^2_k) = x_{a_k} x_{\bar a_k} (v^2_k) = c_k x_{a_k}(v^1_{k+1}) = c_k v^1_k \ne 0.
\]
Continuing in this manner, we see that $x_{\bar
a_{i_1-1}}(v^2_{i_1-1}) \ne 0$ and thus $i_2 < i_1$.

Now, if $k+1 \le j_2$, then
\[
  x_{a_{k+1}} x_{\bar a_{k+1}} (v^2_{k+1}) = x_{\bar a_k} x_{a_k}
  (v^2_{k+1}) = x_{\bar a_k} (v^2_k) \ne 0.
\]
Therefore $x_{\bar a_{k+1}} (v^2_{k+1}) \ne 0$.  But $x_{\bar
a_{k+1}} (v^2_{k+1})$ must be a multiple of $v^1_{k+2}$ as above.
Thus, we must have $k+2 \le j_1$. Continuing in this way, we see
that $j_2 < j_1$. Therefore, in order for $x_{\bar a_k}(v^2_k)$ to
be nonzero for any $k$, we must have $i_2 < i_1$ and $j_2 < j_1$
(for instance, we could have the situation pictured in
Figure~\ref{S:fig:sum-of-two-reps}).

Now, let $x=(x_\Omega, x_{\bar \Omega})$ lie in the conormal bundle
to the point
\[
  x_\Omega = \bigoplus_{l=1}^s x^{i_l, j_l}.
\]
By reordering the indices if necessary, we can assume that $i_1 \ge
i_2 \ge \dots \ge i_s$.  As above, let $V_k^{i_l,j_l} = \C v_k^l$,
for $1 \le l \le s$.  By the above arguments, $x_{\bar a_k}(v^l_k)$
must be a linear combination of $\{v^a_{k+1}\}_{a < l}$.  Thus,
\begin{equation} \label{S:eq:vector-in-kernel}
  v^1_{i_1} \in \ker x_{a_{i_1-1}} \cap \ker x_{\bar a_{i_1}}.
\end{equation}

\begin{exercise} \label{S:ex:stability}
Show that a point $(x,t) \in \Lambda(\bv,\bw)$ satisfies the
stability condition if and only if
\[
  \ker x_{a_{k-1}} \cap \ker x_{\bar a_k} \cap \ker t_k = 0 \quad
  \forall\ 1 \le k \le n.
\]
\end{exercise}

Now consider the lagrangian Nakajima quiver variety
$\mathfrak{L}(\bv,\bw)$ for $\bw=\be^r$ for some $1 \le r \le n$.
Suppose a point $(x,t)$ satisfies the stability condition where $x$
is in the conormal bundle to the point $x_\Omega = \bigoplus_{l=1}^s
x^{i_l, j_l}$.  By Exercice~\ref{S:ex:stability} and
\eqref{S:eq:vector-in-kernel}, we must have $i_1 = r$ and there can
be no other $v^l_k$ in $\ker x_{a_{k-1}} \cap \ker x_{\bar a_k}$ for
any $k$.  By the above considerations, $v^l_{i_l}$ is in $\ker
x_{a_{i_l-1}} \cap \ker x_{\bar a_{i_l}}$ unless $i_l+1=i_{l-1}$.
Thus we have $i_{l+1}=i_l-1$ and $x_{\bar a_{i_{l+1}}} \left(
v^{l+1}_{i_{l+1}} \right) = c_l v^l_{i_l} \ne 0$ for $1 \le l \le
s-1$. Thus, by the above, we must have $j_{l+1} < j_l$ for $1 \le l
\le s-1$.  Thus, the element $x$ can be depicted as in
Figure~\ref{S:fig:x-YD}.

\begin{figure}
\begin{center}
\begin{tikzpicture}
% quiver above
  \draw (-1,0) node {$Q_\Omega$};
  \draw (0,0) node {$\bullet$} node[left] {$\dots$} to (0.5,0);
  \draw (0.5,0) to[<-] (1,0);
  \draw (1,0) node {$\bullet$} node[above] {$r \! - \! 2$} to (1.5,0);
  \draw (1.5,0) to[<-] (2,0);
  \draw (2,0) node {$\bullet$} node[above] {$r \! - \! 1$} to (2.5,0);
  \draw (2.5,0) to[<-] (3,0);
  \draw (3,0) node {$\bullet$} node[above] {$r$} to (3.5,0);
  \draw (3.5,0) to[<-] (4,0);
  \draw (4,0) node {$\bullet$} node[above] {$r \! +\! 1$} to (4.5,0);
  \draw (4.5,0) to[<-] (5,0);
  \draw (5,0) node {$\bullet$} node[above] {$r \! + \! 2$} to (5.5,0);
  \draw (5.5,0) to[<-] (6,0);
  \draw (6,0) node {$\bullet$} to (6.5,0);
  \draw (6.5,0) to[<-] (7,0);
  \draw (7,0) node {$\bullet$} to (7.5,0);
  \draw (7.5,0) to[<-] (8,0) node {$\bullet$} node[right] {$\cdots$};
% first row in Young diagram
  \draw (3,-1) node {$\bullet$} to (3.5,-1);
  \draw (3.5,-1) to[<-] (4,-1);
  \draw (4,-1) node {$\bullet$} to (4.5,-1);
  \draw (4.5,-1) to[<-] (5,-1);
  \draw (5,-1) node {$\bullet$} to (5.5,-1);
  \draw (5.5,-1) to[<-] (6,-1);
  \draw (6,-1) node {$\bullet$} to (6.5,-1);
  \draw (6.5,-1) to[<-] (7,-1);
  \draw (7,-1) node {$\bullet$} to (7.5,-1);
  \draw (7.5,-1) to[<-] (8,-1) node {$\bullet$};
% second row in Young diagram
  \draw (2,-2) node {$\bullet$} to (2.5,-2);
  \draw (2.5,-2) to[<-] (3,-2);
  \draw (3,-2) node {$\bullet$} to (3.5,-2);
  \draw (3.5,-2) to[<-] (4,-2);
  \draw (4,-2) node {$\bullet$} to (4.5,-2);
  \draw (4.5,-2) to[<-] (5,-2);
  \draw (5,-2) node {$\bullet$} to (5.5,-2);
  \draw (5.5,-2) to[<-] (6,-2) node {$\bullet$};
% third row in Young diagram
  \draw (1,-3) node {$\bullet$} to (1.5,-3);
  \draw (1.5,-3) to[<-] (2,-3);
  \draw (2,-3) node {$\bullet$} to (2.5,-3);
  \draw (2.5,-3) to[<-] (3,-3) node {$\bullet$};
% fourth row in Young diagram
  \draw (0,-4) node {$\bullet$};
% Young row diagonal columns
  \draw (0,-4) to[->] (0.5,-3.5);
  \draw (0.5,-3.5) to (1,-3);
  \draw (1,-3) to[->] (1.5,-2.5);
  \draw (1.5,-2.5) to (2,-2);
  \draw (2,-2) to[->] (2.5,-1.5);
  \draw (2.5,-1.5) to (3,-1);
  \draw (2,-3) to[->] (2.5,-2.5);
  \draw (2.5,-2.5) to (3,-2);
  \draw (3,-2) to[->] (3.5,-1.5);
  \draw (3.5,-1.5) to (4,-1);
  \draw (3,-3) to[->] (3.5,-2.5);
  \draw (3.5,-2.5) to (4,-2);
  \draw (4,-2) to[->] (4.5,-1.5);
  \draw (4.5,-1.5) to (5,-1);
  \draw (5,-2) to[->] (5.5,-1.5);
  \draw (5.5,-1.5) to (6,-1);
  \draw (6,-2) to[->] (6.5,-1.5);
  \draw (6.5,-1.5) to (7,-1);
\end{tikzpicture}
\end{center}
\caption{A depiction of an element $x$ of a point $(x,t) \in
\Lambda(\bv,\be^r)$ satisfying the stability condition.
\label{S:fig:x-YD}}
\end{figure}
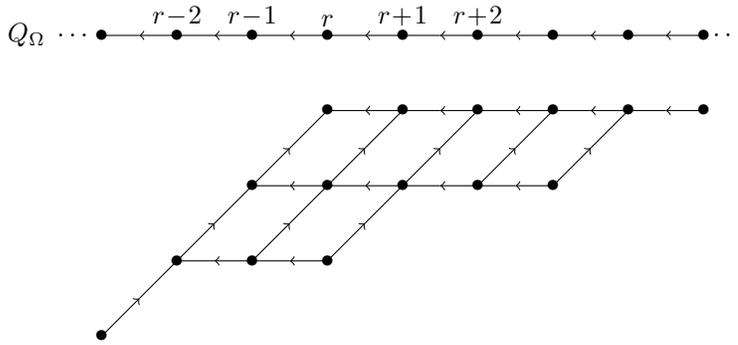

In the tableaux realization of $B(\Lambda_k)$ (see \cite{HK}), the
vertices of the crystal graph consist of single column semistandard
tableaux with $k$ rows and entries from the set $\{1,2, \dots,
n+1\}$.  Note that there is a one-to-one correspondence between this
set of tableaux and the set of Young diagrams with at most $k$ rows
and whose rows have at most $n+1-k$ boxes.  Precisely, the
semistandard tableau
\begin{equation} \label{S:single-column-young-tableaux}
  \parbox{1cm}{\young(\None,\Ntwo,\vdots,\nk)}
\end{equation}
corresponds to the Young diagram (or partition) $(n_k-k,n_{k-1}-1,
\dots, n_1-1)$.  It follows that we have a bijective correspondence
between $B(\omega_k)$ and the set of conormal bundles satisfying the
stability condition, where the tableaux
\eqref{S:single-column-young-tableaux} corresponds to the closure of
the conormal bundle to the orbit through the representation
\[ \textstyle
  \left( \bigoplus_{i=1}^k V^{k+1-i,n_i-1}, \bigoplus_{i=1}^k x^{k+1-i,n_i-1} \right).
\]
Loosely speaking, we have the correspondence

\begin{center}
\begin{tabular}{ccc}
  $\young(j) \quad$ \text{in the $i$th row } & $\quad \longleftrightarrow
  \quad $ & \parbox{5cm}{
  \begin{tikzpicture}[scale=0.5]
  \draw[-<] (0,0) node {$\bullet$} node[below] {$i$} -- (1,0) node[above] {1};
  \draw (1,0) -- (2,0);
  \draw[-<] (2,0) node {$\bullet$} -- (3,0) node[above] {1};
  \draw (3,0) -- (4,0);
  \draw[-] (4,0) node {$\bullet$} -- (5,0) node[right] {$\cdots$};
  \draw (6,0) -- (7,0);
  \draw[-<] (7,0) node {$\bullet$} -- (8,0) node[above] {1};
  \draw (8,0) -- (9,0) node {$\bullet$} node[below] {$j-1$};
  \end{tikzpicture}}
\end{tabular}.
\end{center}

\begin{example} \label{S:eg:diagram-tableaux}
Suppose $\g = \mathfrak{sl}_{10}$ and $k=4$.  Then we have the
following correspondence.
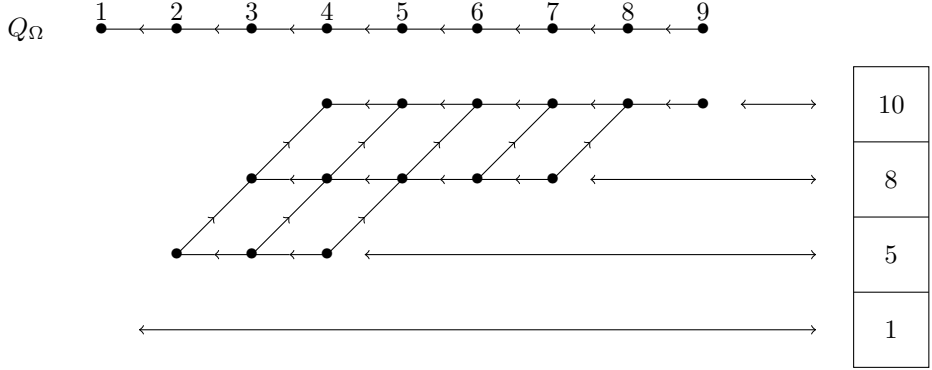
\[
\begin{tikzpicture}
% quiver above
  \draw (-1,0) node {$Q_\Omega$};
  \draw (0,0) node {$\bullet$} node[above] {1} to (0.5,0);
  \draw (0.5,0) to[<-] (1,0);
  \draw (1,0) node {$\bullet$} node[above] {2} to (1.5,0);
  \draw (1.5,0) to[<-] (2,0);
  \draw (2,0) node {$\bullet$} node[above] {3} to (2.5,0);
  \draw (2.5,0) to[<-] (3,0);
  \draw (3,0) node {$\bullet$} node[above] {4} to (3.5,0);
  \draw (3.5,0) to[<-] (4,0);
  \draw (4,0) node {$\bullet$} node[above] {5} to (4.5,0);
  \draw (4.5,0) to[<-] (5,0);
  \draw (5,0) node {$\bullet$} node[above] {6} to (5.5,0);
  \draw (5.5,0) to[<-] (6,0);
  \draw (6,0) node {$\bullet$} node[above] {7} to (6.5,0);
  \draw (6.5,0) to[<-] (7,0);
  \draw (7,0) node {$\bullet$} node[above] {8} to (7.5,0);
  \draw (7.5,0) to[<-] (8,0) node {$\bullet$} node[above] {9};
% first row in Young diagram
  \draw (3,-1) node {$\bullet$} to (3.5,-1);
  \draw (3.5,-1) to[<-] (4,-1);
  \draw (4,-1) node {$\bullet$} to (4.5,-1);
  \draw (4.5,-1) to[<-] (5,-1);
  \draw (5,-1) node {$\bullet$} to (5.5,-1);
  \draw (5.5,-1) to[<-] (6,-1);
  \draw (6,-1) node {$\bullet$} to (6.5,-1);
  \draw (6.5,-1) to[<-] (7,-1);
  \draw (7,-1) node {$\bullet$} to (7.5,-1);
  \draw (7.5,-1) to[<-] (8,-1) node {$\bullet$};
% second row in Young diagram
  \draw (2,-2) node {$\bullet$} to (2.5,-2);
  \draw (2.5,-2) to[<-] (3,-2);
  \draw (3,-2) node {$\bullet$} to (3.5,-2);
  \draw (3.5,-2) to[<-] (4,-2);
  \draw (4,-2) node {$\bullet$} to (4.5,-2);
  \draw (4.5,-2) to[<-] (5,-2);
  \draw (5,-2) node {$\bullet$} to (5.5,-2);
  \draw (5.5,-2) to[<-] (6,-2) node {$\bullet$};
% third row in Young diagram
  \draw (1,-3) node {$\bullet$} to (1.5,-3);
  \draw (1.5,-3) to[<-] (2,-3);
  \draw (2,-3) node {$\bullet$} to (2.5,-3);
  \draw (2.5,-3) to[<-] (3,-3) node {$\bullet$};
% Young row diagonal columns
  \draw (1,-3) to[->] (1.5,-2.5);
  \draw (1.5,-2.5) to (2,-2);
  \draw (2,-2) to[->] (2.5,-1.5);
  \draw (2.5,-1.5) to (3,-1);
  \draw (2,-3) to[->] (2.5,-2.5);
  \draw (2.5,-2.5) to (3,-2);
  \draw (3,-2) to[->] (3.5,-1.5);
  \draw (3.5,-1.5) to (4,-1);
  \draw (3,-3) to[->] (3.5,-2.5);
  \draw (3.5,-2.5) to (4,-2);
  \draw (4,-2) to[->] (4.5,-1.5);
  \draw (4.5,-1.5) to (5,-1);
  \draw (5,-2) to[->] (5.5,-1.5);
  \draw (5.5,-1.5) to (6,-1);
  \draw (6,-2) to[->] (6.5,-1.5);
  \draw (6.5,-1.5) to (7,-1);
% Young tableau
  \draw (10,-0.5) to (10,-4.5) to (11,-4.5) to (11,-0.5) -- cycle;
  \draw (10,-1.5) to (11,-1.5);
  \draw (10,-2.5) to (11,-2.5);
  \draw (10,-3.5) to (11,-3.5);
  \draw (10.5,-1) node {10};
  \draw (10.5,-2) node {8};
  \draw (10.5,-3) node {5};
  \draw (10.5,-4) node {1};
% Correspondence arrows
  \draw[<->] (8.5,-1) to (9.5,-1);
  \draw[<->] (6.5,-2) to (9.5,-2);
  \draw[<->] (3.5,-3) to (9.5,-3);
  \draw[<->] (0.5,-4) to (9.5,-4);
\end{tikzpicture}
\]
Here we draw the tableau upside down in order to make the
correspondence clearer. The tableau
\newcommand{\ten}{10}
\[
  T = \parbox{1cm}{\young(1,5,8,\ten)}
\]
then corresponds to the irreducible component of the lagrangian
Nakajima quiver variety obtained as follows: Let $C_T$ be the
closure of the conormal bundle to the $G_V$-orbit through
\[
  x^{4,9} \oplus x^{3,7} \oplus x^{2,4}.
\]
Then $T$ corresponds to the irreducible component
\[
  \left( \left( C_T \times \sum_{i=1}^n \Hom(V_i,W_i) \right) \cap
  \Lambda(\bv,\bw)^\st \right) / G_V.
\]
In fact, one can show that the Nakajima quiver varieties
$\Lambda(\be^k)$ corresponding to fundamental weights $\omega_k$ are
all single points (this can be shown directly or by using the
dimension formula for Nakaijima quiver varieties and the fact that
these varieties are connected).
\end{example}

What about more general highest weights (that is, highest weights
that are not fundamental weights)?  Consider the highest weight
$\omega_\bw = \sum_{i \in Q_0} w_i \omega_i$.  Then an argument
similar to the above demonstrates that we can have $w_i$ Young
diagrams associated to each vertex $i$.  In order to avoid double
counting, we must associate each irreducible component of the
lagrangian Nakajima quiver variety to a unique such collection of
Young diagrams.  We adopt the convention that larger indecomposable
representations of $Q_\Omega$ (corresponding to the rows of the
Young diagrams) are associated to the Young diagrams attached to
vertices of higher degree (see \cite{Sav03b} for a more precise
treatment). Each Young diagram corresponds to a single column
tableaux as above and, if we organize these columns into a tableau,
our convention corresponds precisely to the condition that this
tableau be semistandard.  We thus obtain a bijective map from the
set $B_g(\bw)$ of irreducible components of the lagrangian Nakajima
quiver variety to the set $B(\omega_w)$ of semistandard tableaux.

\begin{theorem}[{\cite[Theorem~6.4]{Sav03b}}]
\label{S:thm:geom-crystal}
The identification of $B_g(\bw)$ with $B(\omega_w)$ described above
is an isomorphism of crystals.
\end{theorem}

\begin{exercise} \label{S:ex:small-tableaux}
If $\g = \mathfrak{sl}_3$, $\bw = e_1 + e_2$ (so $V(\omega_\bw)$ is
the adjoint representation) and $\bv=e_1 + e_2$ (so
$V(\omega_\bw)_{\omega_\bw - \alpha_\bv} = \mathfrak{h}$ is the zero
weight space), we described the lagrangian Nakajima quiver variety
$\mathfrak{L}(\bv,\bw)$ in Example~\ref{S:eg:NQV-sl3-adjoint}.  It
consists of two projective lines meeting at a point.  Show that the
projective line given by $x_a=0$ corresponds to the tableau
\[
  \young(12,3)
\]
and the projective line given by $\bar x_a = 0$ corresponds to the
tableau
\[
  \young(13,2).
\]
\end{exercise}

\begin{remarks} \label{S:rem:other-types}
\begin{enumerate}
  \item When $\g=\mathfrak{so}_{2n}$ is the simple Lie algebra of type $D_n$,
  a similar explicit correspondence between the irreducible
  components of the lagrangian Nakajima quiver varieties and the tableaux
  appearing in a combinatorial realization of the irreducible
  highest weight representations of $\g$ (see \cite[\S~8.5]{HK} for a description
  of the tableaux appearing in this model) can
  also be given.   However, this description, as well as the tableaux
  appearing in the combinatorial realization, are slightly more
  complicated.  See \cite{Sav03b} for details.

  \item When $\g=\widehat{\mathfrak{sl}}_n$ is the Lie algebra of
  affine type $A$ (see Example~\ref{S:eg:affine-An-quiver}),
  one can
  give a similar explicit enumeration of the irreducible components
  of the lagrangian Nakajima quiver variety.  However, now the ``strings''
  representing indecomposable representations of the quiver
  $Q_\Omega$ can ``wrap around'' the quiver and thus have arbitrary
  length.  For level one representations, we get pictures as in
  Figure~\ref{S:fig:level-one-affine-pic}.
  \begin{figure}
  \begin{center}
  \begin{tikzpicture}
  % Top row
    \draw[-<] (0,0) node {$\bullet$} node[above] {$k$} -- (0.5,0);
    \draw (0.5,0) -- (1,0);
    \draw[-<] (1,0) node {$\bullet$} node[above] {$k\!+\!1$} -- (1.5,0);
    \draw (1.5,0) -- (2,0);
    \draw[-<] (2,0) node {$\bullet$} node[above] {$k\!+\!2$} -- (2.5,0);
    \draw (2.5,0) -- (3,0) node {$\bullet$};
    \draw (3,0) -- (3.5,0) node[right] {$\cdots$};
    \draw (4.5,0) -- (5,0);
    \draw[-<] (5,0) node {$\bullet$} node[above] {$n$} -- (5.5,0);
    \draw (5.5,0) -- (6,0);
    \draw[-<] (6,0) node {$\bullet$} node[above] {0} -- (6.5,0);
    \draw (6.5,0) -- (7,0) node {$\bullet$} node[above] {1};
  % Second row
    \draw[-<] (-1,-1) node {$\bullet$} -- (-0.5,-1);
    \draw (-0.5,-1) -- (0,-1);
    \draw[-<] (0,-1) node {$\bullet$} -- (0.5,-1);
    \draw (0.5,-1) -- (1,-1);
    \draw[-<] (1,-1) node {$\bullet$} -- (1.5,-1);
    \draw (1.5,-1) -- (2,-1) node {$\bullet$};
    \draw (2,-1) -- (2.5,-1) node[right] {$\cdots$};
    \draw (3.5,-1) -- (4,-1) node {$\bullet$};
  % Third row
    \draw[-<] (-2,-2) node {$\bullet$} -- (-1.5,-2);
    \draw (-1.5,-2) -- (-1,-2);
    \draw[-<] (-1,-2) node {$\bullet$} -- (-0.5,-2);
    \draw (-0.5,-2) -- (0,-2) node {$\bullet$};
  % Young row diagonal columns
    \draw[->] (-2,-2) -- (-1.5,-1.5);
    \draw (-1.5,-1.5) -- (-1,-1);
    \draw[->] (-1,-1) -- (-0.5,-0.5);
    \draw (-0.5,-0.5) -- (0,0);
    \draw[->] (-1,-2) -- (-0.5,-1.5);
    \draw (-0.5,-1.5) -- (0,-1);
    \draw[->] (0,-1) -- (0.5,-0.5);
    \draw (0.5,-0.5) -- (1,0);
    \draw[->] (0,-2) -- (0.5,-1.5);
    \draw (0.5,-1.5) -- (1,-1);
    \draw[->] (1,-1) -- (1.5,-0.5);
    \draw (1.5,-0.5) -- (2,0);
    \draw[->] (2,-1) -- (2.5,-0.5);
    \draw (2.5,-0.5) -- (3,0);
    \draw[->] (4,-1) -- (4.5,-0.5);
    \draw (4.5,-0.5) -- (5,0);
  \end{tikzpicture}
  \end{center}
  \caption{A diagram representing a point in the quiver variety
  corresponding to a level one representation (of highest weight
  $\omega_k$) of the Lie algebra $\widehat{\mathfrak{sl}}_n$
  \label{S:fig:level-one-affine-pic}}
  \end{figure}
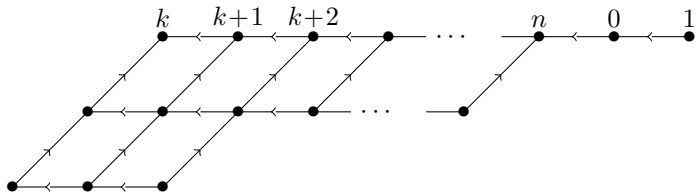
  If we straighten this picture, flip it (in the line $x=y$ if we view
  the picture as lying in the standard $(x,y)$-plane), and replace
  vertices by boxes, we obtain Figure~\ref{S:fig:Young-wall},
  \begin{figure}
  \begin{center}
  \begin{tikzpicture}
  % Vertical lines
    \draw (0,0) -- (0,4.2);
    \draw[loosely dotted] (0,4.2) -- (0,4.8);
    \draw (0,4.8) -- (0,8);
    \draw (-1,0) -- (-1,4.2);
    \draw[loosely dotted] (-1,4.2) -- (-1,4.8);
    \draw (-1,4.8) -- (-1,8);
    \draw (-2,0) -- (-2,4.2);
    \draw[loosely dotted] (-2,4.2) -- (-2,4.8);
    \draw (-2,4.8) -- (-2,6);
    \draw (-3,0) -- (-3,3);
  % Horizontal lines
    \draw (-3,0) -- (0,0);
    \draw (-3,1) -- (0,1);
    \draw (-3,2) -- (0,2);
    \draw (-3,3) -- (0,3);
    \draw (-2,4) -- (0,4);
    \draw (-2,5) -- (0,5);
    \draw (-2,6) -- (0,6);
    \draw (-1,7) -- (0,7);
    \draw (-1,8) -- (0,8);
  % Entries
    \draw (-0.5,0.5) node {$k$};
    \draw (-0.5,1.5) node {$k+1$};
    \draw (-0.5,2.5) node {$k+2$};
    \draw (-0.5,3.5) node {$k+3$};
    \draw (-0.5,5.5) node {$n$};
    \draw (-0.5,6.5) node {$0$};
    \draw (-0.5,7.5) node {$1$};
    \draw (-1.5,0.5) node {$k-1$};
    \draw (-1.5,1.5) node {$k$};
    \draw (-1.5,2.5) node {$k+1$};
    \draw (-1.5,3.5) node {$k+2$};
    \draw (-1.5,5.5) node {$n-1$};
    \draw (-2.5,0.5) node {$k-2$};
    \draw (-2.5,1.5) node {$k-1$};
    \draw (-2.5,2.5) node {$k$};
  \end{tikzpicture}
  \end{center}
  \caption{The Young wall corresponding to the quiver representation of
  Figure~\ref{S:fig:level-one-affine-pic}. \label{S:fig:Young-wall}}
  \end{figure}
  which is simply one of the Young walls appearing in
  \lecturecite{Section~\ref{K:section_Young}}{\cite[Section~7]{Kang}}.
  Since we are no longer in the finite type case, we must impose the
  nilpotency condition.  This corresponds to the condition that the
  Young walls be proper.

  \item  In the case where $\g=\widehat{\mathfrak{sl}}_n$ and
  $\omega_\bw = \sum_{i=0}^n w_i \omega_i$
  is arbitrary (i.e. we consider general dominant integral weights),
  we get $w_i$ Young walls associated to each vertex $i$ as in the
  case of $\mathfrak{sl}_n$.  Again, choosing a convention to avoid
  double counting becomes a condition on the heights of columns
  appearing. What remains are combinatorial objects called \emph{Young
  pyramids}. See Figure~\ref{S:fig:Young-pyramid} for an example of a
  Young pyramid.  Further details on these objects and their
  connections to quiver varieties can be found in \cite{Sav04b}.

  \begin{figure}
  \begin{center}
  \begin{tikzpicture}[scale=0.7]
    \planepartition{{1,1,5,5,4,3,2,2},{1,1,4,4,4,2,1},{1,3,3,2,2,1},{1,2,2,1},{2,1}}
  \end{tikzpicture}
  \end{center}
  \caption{A Young pyramid corresponding to the highest weight
  $2\omega_0 + 2 \omega_1 + \omega_2$.  We think of each vertical wall
  running diagonally from top-left to bottom-right as a Young diagram
  or Young wall. The positions of the leftmost columns of these walls
  are determined by the highest weight.  Here the leftmost columns of
  the back two walls are aligned (corresponding to the summand of
  $2\omega_0$), the next two walls are aligned one position to the
  left (corresponding to the summand of $2\omega_1$) and the frontmost
  wall starts one position further to the left (corresponding to
  the summand $\omega_2$). \label{S:fig:Young-pyramid}}
  \end{figure}
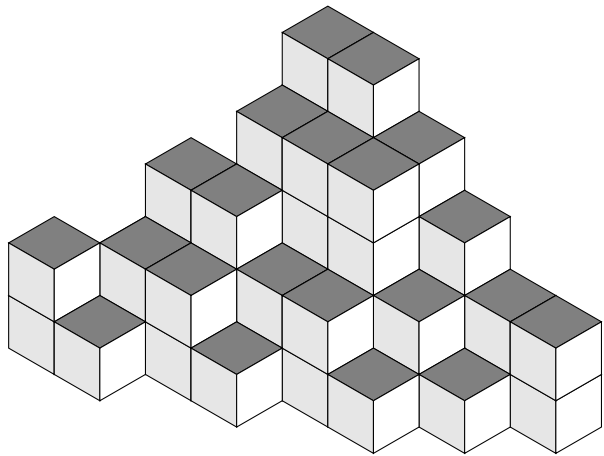
\end{enumerate}
\end{remarks}

%%%%%%%%%%%%%%%%%%%%%%%%%%%%%%%%%%%%%%%%%%%%%%%%%%%%%%%%%%%%%%%%%%%%
%
% References and End Matter
%
%%%%%%%%%%%%%%%%%%%%%%%%%%%%%%%%%%%%%%%%%%%%%%%%%%%%%%%%%%%%%%%%%%%%

\bibliographystyle{abbrv}
\bibliography{biblist}

\vspace{4mm} \noindent \small{Alistair Savage: {\it Department of
Mathematics and Statistics, University of Ottawa, 585 King Edward
Ave, Ottawa, ON, Canada K1N 6N5.}}

\end{document}